\documentclass[a4paper, 10pt, twoside, notitlepage]{amsart}

\usepackage{amsmath,amscd}
\usepackage{amssymb}
\usepackage{amsthm}
\usepackage{comment}
\usepackage{graphicx}

\usepackage{mathrsfs}
\usepackage[ocgcolorlinks, linkcolor=blue]{hyperref}

\usepackage[ocgcolorlinks,linkcolor=blue]{hyperref}

\theoremstyle{plain}
\newtheorem{thm}{Theorem}[section]
\newtheorem{prop}{Proposition}[section]
\newtheorem{lem}[prop]{Lemma}
\newtheorem{cor}[prop]{Corollary}

\newtheorem{defi}[prop]{Definition}
\newtheorem{rmk}[prop]{Remark}

\numberwithin{equation}{section}
\newcommand {\R} {\mathbb{R}} 
 \newcommand {\N} {\mathbb{N}}
\newcommand {\C} {\mathbb{C}} 
\newcommand {\p} {\partial}

\newcommand{\eps}{\epsilon}

\newcommand{\ol}{\overline}

\newcommand{\mR}{\mathbb{R}}                    
\newcommand{\mC}{\mathbb{C}}                    
\newcommand{\mZ}{\mathbb{Z}}                    
\newcommand{\abs}[1]{\lvert #1 \rvert}          
\newcommand{\norm}[1]{\lVert #1 \rVert}         
\newcommand{\re}{\mathrm{Re}}
\newcommand{\im}{\mathrm{Im}}

\DeclareMathOperator{\Id} {Id}

\title[]{Inverse problems for elliptic equations with power type nonlinearities}

\dedicatory{Dedicated to the memory of Yaroslav Kurylev}

\author[Lassas]{Matti Lassas}
\address{Department of Mathematics and Statistics, University of Helsinki}
\curraddr{}
\email{matti.lassas@helsinki.fi}

\author[Liimatainen]{Tony Liimatainen}
\address{Department of Mathematics and Statistics, University of Jyv\"askyl\"a}
\curraddr{}
\email{tony.liimatainen@helsinki.fi}

\author[Lin]{Yi-Hsuan Lin}
\address{Department of Mathematics and Statistics, University of Jyv\"askyl\"a}
\curraddr{}
\email{yihsuanlin3@gmail.com}

\author[Salo]{Mikko Salo}
\address{Department of Mathematics and Statistics, University of Jyv\"askyl\"a}
\curraddr{}
\email{mikko.j.salo@jyu.fi}

\begin{document}
	
	\maketitle

	\begin{abstract}
	We introduce a method for solving Calder\'on type inverse problems for semilinear equations with power type nonlinearities. The method is  based on higher order linearizations, and it allows one to solve inverse problems for certain nonlinear equations in cases where the solution for a corresponding linear equation is not known. Assuming the knowledge of a nonlinear Dirichlet-to-Neumann map, we determine both a potential and a conformal manifold simultaneously in dimension $2$, and a potential on transversally anisotropic manifolds in dimensions $n \geq 3$. In the Euclidean case, we show that one can solve the Calder\'on problem for certain semilinear equations in a surprisingly simple way without using complex geometrical optics solutions. 
	
		
		\medskip
		
		
		
	\end{abstract}

	\tableofcontents

	\section{Introduction}
	In this paper we study inverse boundary value problems for nonlinear elliptic equations. 
	A standard example of inverse problems for linear elliptic equations is the problem introduced by Calder\'on~\cite{calderon2006inverse}, where the objective is to determine the electrical conductivity of a medium by making voltage and current measurements on its boundary. It is closely related to the problem of determining an unknown potential $q$ in a Schr\"odinger operator $\Delta +q$ from boundary measurements, first solved in \cite{sylvester1987global} in dimensions $n \geq 3$. There is an extensive theory concerning inverse boundary value problems for linear elliptic equations, and we refer to~\cite{uhlmann2009calderon} for a survey.
	
	It is also natural to consider analogous inverse problems under nonlinear settings. Let $\Omega \subset \mR^n$ be a bounded domain with $C^\infty$ boundary, and consider the reaction-diffusion equation 
	\[
	\p_t w - \Delta w = a(x,w) \text{ in $\Omega \times \{t > 0 \}$}.
	\]
	Equations of this type arise in the modelling of chemical reactions, population dynamics and pattern formation \cite{volpert2014elliptic}. 
	Examples include the Fisher, Kolmogorov or logistic diffusion equations with quadratic nonlinearity (i.e.\ $a(x,w)$ is quadratic in $w$), the Newell-Whitehead-Segel equation with cubic nonlinearity, and equations in combustion involving polynomial or exponential nonlinearities.
	
	A stationary solution $w(x,t) = u(x)$ satisfies the elliptic equation 
	\[
	\Delta u + a(x,u) = 0 \text{ in $\Omega$}.
	\]
	The Dirichlet problem for this equation is related to maintaining a temperature (or concentration or population) $f$ on the boundary. The boundary measurements for such an equation, provided that it is well-posed for some class of boundary values, may be encoded by a Dirichlet-to-Neumann map (DN map) $\Lambda_q$, which maps the boundary value $f$ to the flux $\Lambda_q(f) = \p_{\nu} u|_{\p \Omega}$ of the corresponding equilibrium state across the boundary.
	
	In fact, inverse problems for nonlinear elliptic equations have also been widely studied. A standard method, introduced in \cite{isakov1993uniqueness} in the parabolic case, is to show that the first linearization of the nonlinear DN map is actually the DN map of a linear equation, and to use the theory of inverse problems for linear equations. For the semilinear Schr\"{o}dinger equation $\Delta u +a(x,u)=0$, the problem of recovering the potential $a(x,u)$ was studied in \cite{isakov1994global, sun2010inverse} in dimensions $n\geq 3$, and in \cite{VictorN, sun2010inverse, imanuvilovyamamoto_semilinear} when $n=2$. In addition, inverse problems have been studied for quasilinear elliptic equations \cite{sun1996, sun1997inverse, kang2002identification, liwang_navierstokes, munozuhlmann}, the degenerate elliptic $p$-Laplace equation \cite{salo2012inverse, branderetal_monotonicity_plaplace}, and the fractional semilinear Schr\"odinger equation \cite{lai2019global}. Certain Calder\'on type inverse problems for quasilinear equations on Riemannian manifolds were recently considered in~\cite{lassas2018poisson}. We refer to the survey articles \cite{sun2005, uhlmann2009calderon} for further details on inverse problems for nonlinear elliptic equations. 
	
	Inverse problems have also been studied for hyperbolic equations with various nonlinearities.  Many of the works mentioned above rely on a solution to a related inverse problem for a linear equation. This is in contrast to the study of inverse problems for nonlinear hyperbolic equations, where it has been realized that the nonlinearity can actually be beneficial in solving inverse problems. 
	
	By using the nonlinearity as a tool, some still unsolved inverse problems for hyperbolic linear equations have been solved for their nonlinear counterparts. For the scalar wave equation with a quadratic nonlinearity, Kurylev-Lassas-Uhlmann \cite{kurylev2018inverse} proved that local measurements determine the global topology, differentiable structure and the conformal class of the metric $g$ on a globally hyperbolic $4$-dimensional Lorentzian manifold. The authors of \cite{lassas2018inverse} studied inverse problems for general semilinear wave equations on Lorentzian manifolds, and in \cite{lassas2017determination} they studied analogous problem for the Einstein-Maxwell equations. For more inverse problems of nonlinear hyperbolic equations, we refer readers to \cite{chen2019detection,de2018nonlinear,kurylev2014einstein,wang2016quadartic} and references there in.
	
	In this work we introduce a method which uses nonlinearity as a tool that helps in solving inverse problems for certain nonlinear elliptic equations. The method is based on \emph{higher order linearizations} of the DN map, and essentially amounts to using sources with several parameters and obtaining new linearized equations after differentiating with respect to these parameters. We demonstrate the scope of the method by solving Calder\'on type problems for three mathematical models. 
	
	The first model is the Calder\'on problem for a semilinear Schr\"odinger equation with quadratic nonlinearity, 
	\begin{equation}\label{usquare}
	 \Delta u + qu^2=0 \text{ in } \Omega\subset \R^n,
	\end{equation}
    where $q \in C^{\infty}(\ol{\Omega})$ and $n\geq 2$. The solution to a related inverse problem with $a(x,u)$ in place of $qu^2$ is known under assumptions like $\p_u a(x,u) \leq 0$~\cite{isakov1994global, VictorN, sun2010inverse}. Theorem \ref{main thm q} proves uniqueness for the nonlinearity $qu^2$, which appears to be a new result. The method applies to more general models, but we begin with the operator~\eqref{usquare} in order to 
	introduce our approach in the simplest possible setting. 
	
	The second new result is Theorem \ref{main thm 2D}, where we simultaneously determine the metric, the manifold and the potential up to gauge symmetry from the knowledge of the DN map of a semilinear Schr\"odinger equation on two-dimensional Riemannian surfaces. The analogous result for a linear Schr\"odinger equation is not known in this generality. Here we use nonlinearity to simultaneously determine the topology and the conformal structure of the Riemannian surface, as well as the potential, up to a natural gauge transformation. 
	
	The third result, Theorem \ref{thm for CTA}, is the recovery of the potential $q$ from the knowledge of the DN map of a Schr\"odinger operator with nonlinearity of the form $qu^m$, $m\geq 3$, on  transversally anisotropic manifolds in dimensions $n \geq 3$. Transversally anisotropic manifolds are product type manifolds which appear in several works related to the anisotropic Calder\'on problem. Again, the solution to the analogous inverse problem for a linear equation is not known in this generality. Existing results will be discussed in more detail later in this introduction.
	
	Let us introduce  the mathematical setting for this article. We will denote by $(M,g)$ a compact Riemannian manifold with $C^\infty$ boundary $\p M$, where $\dim(M)=n$, $n\geq 2$. For example, one could have $M = \ol{\Omega}$ where $\Omega$ is a bounded $C^{\infty}$ domain in $\R^n$, and $g$ could be the Euclidean metric. Let $q \in C^{\infty}(M)$. We will consider semilinear elliptic equations of the form 
	
	\begin{align}\label{Main semilinear equation}
	\begin{cases}
	\Delta_g u +q u^m =0 & \text{ in }M, \\
	u=f & \text{ on }\p M,
	\end{cases}
	\end{align}
	where 
	\[
	m\in \N \text{ and } m\geq 2.
	\]
 Here $\Delta_g$ is the Laplace-Beltrami operator, given in local coordinates by 
	\[
	\Delta_g u=\frac{1}{\det(g)^{1/2}}\sum_{a,b=1}^n\frac{\p}{\p x_a}\left( \det(g)^{1/2} g^{ab}\frac{\p u}{\p x_b}\right),
	\]
where $g=(g_{ab}(x))$ and $(g^{ab}(x))=g^{-1}$.

	We will show that the Dirichlet problem \eqref{Main semilinear equation} has a unique small solution $u$ for sufficiently small boundary data $f\in C^s(\p M)$, where $s>1$ with $s\notin \N$. More precisely this means that there is $\delta>0$ such that whenever $\norm{f}_{C^s(\p M)}\leq \delta$ , there is a unique solution $u_f$ to \eqref{Main semilinear equation} with sufficiently small $C^s(M)$ norm (see Section \ref{Section 2} for more details on well-posedness). We will call $u_f$ the unique small solution. Here $C^s$ is the standard H\"older space for $s>1$ with $s\notin \N$ (often written as $C^{k,\alpha}$ if $s = k+\alpha$ where $k \in \mZ$ and $0 < \alpha < 1$), see e.g.\ \cite[Section 13.8]{taylor2011partial}. Hence, the DN map is defined by using the unique small solution in a following way:
	\begin{align*}
	\Lambda_{M,g,q}: C^s(\p M)\to C^{s-1}(\p M), \ \  f \mapsto \p_{\nu} u_f|_{\p M},
	\end{align*}
	where $\p_{\nu}$ denotes the normal derivative on the boundary $\p M$.
	
	
	As a warm-up, we begin with a theorem that illustrates our method in a simple setting. This theorem is in $\R^n$ for $n\geq 2$, where $\Delta_g$ is the Euclidean Laplacian and $M = \ol{\Omega}$ with $\Omega $ a bounded smooth domain in $\R^n$.
	\begin{thm}[Global uniqueness for a quadratic nonlinearity]\label{main thm q}
		Let $n\geq 2$, and let $\Omega\subset \R^n$ be a bounded domain with $C^\infty$ boundary $\p \Omega$. Let $q_1,q_2\in C^\infty(\ol{\Omega})$. Assume the DN maps $\Lambda_{q_j}$ for the equations 
		\begin{align}\label{simplest_nonlinear}
		\begin{cases}
		  \Delta u + q_ju^2 =0 &\text{ in } \Omega, \\
		u=f  &\text{ on } \p \Omega,
		\end{cases}
		\end{align}
for $j=1,2$ satisfy 
		\[
		\Lambda_{q_1}(f)=\Lambda_{q_2}(f) 
		\]
		for all $f\in C^s(\p \Omega)$ with $\norm{f}_{C^s(\p M)}<\delta$, where $\delta>0$ is any sufficiently small number. Then $q_1=q_2$ in $\Omega$.
	\end{thm}

	
	We will offer a detailed proof of Theorem \ref{main thm q} in Section~\ref{Section 3}, but let us briefly discuss the idea how to prove the theorem by using the method of \emph{higher order linearization}. The second order linearization of the nonlinear DN map has already been used in the works \cite{sun1996, sun1997inverse} related to nonlinear equations with matrix coefficients. First and second order linearizations were also used in in \cite{kang2002identification} for a nonlinear conductivity equation (see also \cite{carsteanakamuravashisth}). Under certain assumptions on the nonlinearity, by using the second order linearization, they can recover quadratic parts of the nonlinearity (see \cite[Theorem 1.2 and Theorem 1.3]{kang2002identification}). In this work, we use similar ideas but obtain interesting new phenomena for related nonlinear inverse problems.
	
	For the equation \eqref{simplest_nonlinear} with quadratic nonlinearity, the first order linearization of the nonlinear DN map $\Lambda_q$, linearized at the zero boundary value, is just the DN map for the standard Laplace equation:
	\[
	(D\Lambda_q)_0: C^s(\p \Omega) \to C^{s-1}(\p \Omega), \ \ f \mapsto \p_{\nu} v_f|_{\p \Omega},
	\]
	where $v_f$ is the unique solution of $\Delta v_f = 0$ in $\Omega$ with $v_f|_{\p \Omega} = f$. Thus the first linearization does not carry any information about the unknown potential $q$. However, for a quadratic nonlinearity the \emph{second linearization} $(D^2 \Lambda_q)_0$, which is a symmetric bilinear map from $C^s(\p M) \times C^s(\p M)$ to $C^{s-1}(\p M)$, turns out to be very useful: it is characterized by the identity (see \eqref{dm_lambdaq_identity})  
	\[
	\int_{\p \Omega} (D^2 \Lambda_q)_0(f_1, f_2) f_3 \,dS = -2 \int_{\Omega} q v_{f_1} v_{f_2} v_{f_3} \,dx
	\]
	where $v_{f_j}$ is the harmonic function with boundary value $f_j$. Thus we have the implications 
	\begin{align*}
	 &\Lambda_{q_1}(f) = \Lambda_{q_2}(f) \text{ for small $f$} \\
	 \implies &(D^2 \Lambda_{q_1})_0 = (D^2 \Lambda_{q_2})_0 \\
	 \implies &\int_{\Omega} (q_1-q_2) v_1 v_2 v_3 \,dx = 0
	\end{align*}
	for any functions $v_1, v_2, v_3 \in C^s(\ol{\Omega})$ that are harmonic in $\Omega$.
	
	The last statement is very close to the linearized Calder\'on problem for a \emph{linear} Schr\"odinger equation (the difference is that here one has the product of three harmonic functions, instead of two). Choosing $v_1$ and $v_2$ to be harmonic exponentials as in the work of Calder\'on \cite{calderon2006inverse}, and choosing $v_3 \equiv 1$, shows that the Fourier transform of $q_1-q_2$ vanishes and hence $q_1 = q_2$. Thus, somewhat strikingly, we can solve a Calder\'on type inverse problem for the nonlinear equation $\Delta u + qu^2 = 0$ in a much simpler way than for the linear equation $\Delta u + qu=0$ (the latter requires complex geometrical optics solutions as in  \cite{sylvester1987global}). The method also provides extremely simple reconstruction of the potential $q$, see Corollary~\ref{cor main thm q}.
	
	We also mention that the second order linearization can be described as 
	\[
	(D^2 \Lambda_q)_0(f_1, f_2) = \p_{\eps_1} \p_{\eps_2} u_{\eps_1 f_1 + \eps_2 f_2}|_{\eps_1=\eps_2=0}.
	\]
	That is, one considers boundary data 
	\[
	f=\eps_1 f_1 + \eps_2 f_2\in C^s(\p \Omega),
	\]
	where $\eps_1,\eps_2$ are sufficiently small parameters, and takes the mixed derivative 
\[
                                              \left.\frac{\p}{\p \eps_1}\frac{\p}{\p \eps_2}\right|_{\eps_1=\eps_2=0}
\]
of the equation~\eqref{simplest_nonlinear}. This idea is similar to the recent works on inverse problems for nonlinear hyperbolic equations mentioned above, and it yields the equations 
    \begin{equation}\label{second_deriv}
     \Delta w_j = -2q_j v_{f_1} v_{f_2},
    \end{equation}
    for $j=1,2$, where $w_j =   \left.\frac{\p}{\p \eps_1}\frac{\p}{\p \eps_2}\right|_{\eps_1=\eps_2=0} u_j$ and $v_{f_j}$ are harmonic functions, i.e.\ solutions to the linearized equation $\Delta v = 0$. 
    Taking the mixed derivative of the DN maps yields (see Section \ref{Section 2})
    \begin{equation*}\label{normal_derivs_agree}
     \p_\nu w_1=\p_\nu w_2 \text{ on } \p \Omega.
    \end{equation*}
    Subtracting the equations~\eqref{second_deriv} for $j=1,2$ and integrating the resulting equation against the harmonic function $v_{f_3}$ yields the desired formula 
    \[
    \int_{\Omega} (q_1-q_2) v_{f_1} v_{f_2} v_{f_3} \,dx = 0
    \]
    which was mentioned in the discussion above.

We move on to describe our next result. By using higher order linearizations we prove the following \emph{simultaneous recovery} on a two-dimensional Riemannian surface. 

	\begin{thm}[Simultaneous recovery  of metric and potential]\label{main thm 2D}
	Let  $(M_1,g_1)$ and $(M_2,g_2)$ be two compact connected manifolds with mutual $C^\infty$ boundary $\p M$, where $\dim (M_1)=\dim (M_2) =2$ and $m\geq 2$. Let $\Lambda_{M_j,g_j,q_j}$ be the DN maps of 
	\begin{align}\label{2Deq}
	\Delta_{g_j} u +q_j u^m=0  \text{ in }M_j
	\end{align}
	for $j=1,2$.
	Let $s>1$ with $s \notin \N$ and assume that 
	$$\Lambda_{M_1,g_1,q_1} f =\Lambda_{M_2,g_2,q_2} f \text{ on }\p M, $$ for any $f\in C^s(\p M)$ with $\norm{f}_{C^s(\p M)}\leq \delta$, where $\delta >0$ is sufficiently small. Then:
	\begin{itemize}
		\item[(1)] There exists a conformal diffeomorphism $J:M_1\to M_2$ and a positive smooth function $\sigma$ such that 
		\begin{equation*}\label{J_and_sigma}
		\sigma J^*g_2 = g_1, \text{ with }J|_{\p M}=\Id \text{ and }\sigma|_{\p M} =1.
		\end{equation*}
		
		\item[(2)] Moreover, one can also recover the potential up to a natural gauge invariance in the sense that
		\[
		\sigma q_1 = q_2 \circ J \text{ in }M_1.
		\]
	\end{itemize}
	\end{thm}
	
	We see that the conformal factor $\sigma$ (and also the diffeomorphism $J$) couples to the potential. This is due to the \emph{gauge symmetry} of the inverse problem: 
	\[
	 \Lambda_{M_1,\sigma J^*g,\sigma^{-1}J^*q}=\Lambda_{J(M_1),g,q}
	\]
	where $J$ and $\sigma$ satisfy boundary conditions as above.
        For the linear equation $\Delta_g u + qu = 0$, an analogous result has been proved when $M$ is a domain in $\mR^2$ with a Riemannian metric \cite{iuy_general}, when $M$ is a manifold and the potentials are zero \cite{lassas2001determining}, and when the manifold $M$ is a priori known \cite{guillarmou2011calderon}. The recovery of properties of both the manifold and potential is stated as an open question in \cite{guillarmoutzou2013_survey}, where further references to two-dimensional results are given. The proof of Theorem \ref{main thm 2D} uses the first linearization of the DN map to recover the metric and the manifold up to a conformal transformation. Then the second linearization is used to recover the potential on a single fixed manifold (up to the gauge symmetry). 

	
	The final new result in this article is to consider inverse problems for the semilinear Schr\"odinger equation on \emph{transversally anisotropic} manifold. Let us recall the definition of a transversally anisotropic manifold.
	\begin{defi}
		Let $(M,g)$ be a compact oriented manifold with a $C^\infty$ boundary and with $\dim M\geq 3$.
			$(M,g)$ is called \emph{transversally anisotropic} if $(M,g)\subset \subset  (T,g)$, where $T=\R \times M_0$ and $g=e\oplus g_0$. Here $(\R, e)$ is the Euclidean line and $(M_0, g_0)$ is $(n-1)$-dimensional compact manifold with a smooth boundary.
			
			
	\end{defi}
	\noindent For more details of inverse problems in transversally anisotropic geometries for linear equations, we refer readers to \cite{ferreira2013calderon,ferreira2017linearized}. 
	
	We prove the following.
	\begin{thm}\label{thm for CTA}
		Let $(M,g)$ be a transversally anisotropic manifold, let $q_j \in C^{\infty}(M)$, and let $\Lambda_{q_j}$ be the DN maps for the equations 
		\[
		\Delta_g u + q_j u^m =0 \text{ in }M
		\]
		for $j=1,2$, where we assume that 
		\[
		m\in \N, \quad m\geq 3. 
		\]
		If the DN maps satisfy 
		\[
		\Lambda_{q_1}(f) =\Lambda_{q_2}(f)
		\]
		for all sufficiently small $f$, then $q_1=q_2$ in $M$.
	\end{thm}
	
	The higher order linearization method in this case reduces the proof of Theorem~\ref{thm for CTA} to showing for any $m \geq 3$ that the identity  
	\begin{equation}\label{mprod}
	\int_M f v_1 \cdots v_{m+1} \,dV = 0
	\end{equation}
	holding for any $v_j \in C^{\infty}(M)$ with $\Delta_g v_j = 0$ in $M$, implies $f \equiv 0$. Thus we prove that the products of at least four harmonic functions on a transversally anisotropic manifold form a complete set. The main point is that the argument works for arbitrary transversally anisotropic manifolds without any restriction on the transversal geometry.
	
	The solution to the analogous inverse problem for a linear equation $\Delta_g u + qu = 0$ on transversally anisotropic manifolds is only known under the additional assumption that the transversal manifold $(M_0,g_0)$ has injective geodesic X-ray transform~\cite{ferreira2013calderon}. In the linearized version of that problem, the identity~\eqref{mprod} only holds for $m=1$ and one needs to prove that products of pairs of harmonic functions form a complete set. In~\cite{ferreira2013calderon} this is done by using complex geometrical optics solutions that concentrate near two-dimensional surfaces that are translates of geodesics on $M_0$. Using products of such solutions and their complex conjugates recovers certain integrals over geodesics in $M_0$, but does not yield pointwise information. In~\cite{ferreira2017linearized} products of solutions concentrating near two intersecting geodesics were used instead to recover microlocal information in the linearized inverse problem. The products are supported near finitely many points in $M_0$, but there is oscillation that prevents recovering more information. We also mention \cite{guillarmousalotzou_complex} that deals with the linearized problem on certain complex manifolds.
	
	The idea behind the proof of Theorem \ref{thm for CTA} is that since one can use products of at least four harmonic functions, we can use solutions related to two intersecting geodesics on $M_0$ as well as their complex conjugates. The product of these four solutions is supported near finitely many points in $M_0$ and the product does not have high oscillations. This allows one to recover the potential completely.
	
	We mention that the aim of this paper is not to work in the highest possible generality or to provide an extensive list of all possible applications of the higher order linearization method. For example, it is clear that the method applies to certain more general nonlinearities and less regular coefficients. These are left to forthcoming works. Here we have included applications that illustrate the power of the higher order linearization method.

	Finally, we mention that before submitting this paper we became aware of an upcoming preprint of Ali Feizmohammadi and Lauri Oksanen, which simultaneously and independently proves a result similar to Theorem \ref{thm for CTA}, and we agreed with them to publish the preprints of the results at the same time on the same preprint server.
	
	The paper is organized as follows. In Section \ref{Section 2}, we lay out the basic properties for semilinear elliptic equations that we use. This includes the well-posedness of the Dirichlet problem and higher order linearizations of the DN map. 
	We use the higher order linearization approach to prove Theorem \ref{main thm q} in Section \ref{Section 3}, Theorem \ref{main thm 2D} in Section \ref{Section 4}, and Theorem \ref{thm for CTA} in Section \ref{Section 5}, respectively. 
	
\vspace{10pt}

\noindent {\bf Acknowledgements.}
All authors were supported by the Finnish Centre of Excellence in Inverse Modelling and Imaging (Academy of Finland grant 284715). M.S.\ was also supported by the Academy of Finland (grant 309963) and by the European Research Council under Horizon 2020 (ERC CoG 770924).

\section{Preliminaries}\label{Section 2}
	
In this section, we prove well-posedness of the Dirichlet problem for semilinear elliptic equations with small boundary data, and study higher order linearizations of the DN map. We state the first result for a general nonlinearity satisfying two conditions: the first ensures that $u \equiv 0$ is a solution, and the second states that the equation linearized at $u \equiv 0$ is well-posed. 



\begin{prop}[Well-posedness]\label{wellposedness_and_expansion}
	Let $(M,g)$ be a compact Riemannian manifold with $C^\infty$ boundary $\p M$ and let $Q$ be the semilinear elliptic operator 
	\[
	Q(u):= \Delta_g u + a(x,u),
	\]
	where $a \in C^{\infty}(M \times \mR)$ satisfies the following two conditions:
	\begin{gather}
	a(x,0) = 0, \label{a_first_condition} \\
	\text{The map $v \mapsto \Delta_g v + \p_u a(\,\cdot\,,0)v$ is injective on $H^1_0(M)$.} \label{a_second_condition}
	\end{gather}
    Let $s > 1$ with $s \notin \mZ$.
	There exist $\delta, C > 0$ such that for any $f$ in the set 
	\[
	U_{\delta} := \{ h \in C^s(\p M) \,;\, \norm{h}_{C^s(\p M)}<\delta \},
	\]
    there is a solution $u = u_f$ of 
	\begin{align}\label{solvability}
	\begin{cases}
	\Delta_g u+a(x,u)=0 & \text{ in } M,\\
	u= f & \text{ on } \p M,
	\end{cases}
	\end{align}
	which satisfies
	\begin{align*}\label{Cont_of_sols}
	\norm{u}_{C^s(M)}\leq C \norm{f}_{C^s(\p M)}.
	\end{align*}
	The solution $u_f$ is unique within the class $\{ w \in C^s(M) \,;\, \norm{w}_{C^s(M)} \leq C \delta \}$, and if $f \in C^{\infty}(\p M)$, then $u_f \in C^{\infty}(M)$. Moreover, there are $C^{\infty}$ maps 
	\begin{align*}
	\begin{array}{rll}
	S:& \!\!\!U_{\delta} \to C^s(M), \ & f \mapsto u_f, \\
	\Lambda:& \!\!\!U_{\delta} \to C^{s-1}(\p M), \ &f \mapsto \p_{\nu} u_f|_{\p M}.
	\end{array}
	\end{align*}
\end{prop}
\begin{proof}
	We prove the existence of solutions by using the implicit function theorem in Banach spaces~\cite[Theorem 10.6]{renardy2006introduction}. Let 
	\[
	 X=C^s(\p M), \quad Y=C^{s}(M), \quad Z=C^{s-2}(M)\times C^s(\p M).
	\]
	Consider the map 
	\[
	 F:X\times Y\to Z, \quad F(f,u)=(Q(u),u|_{\p M}-f).
	\]
	We wish to show that $F$ indeed maps to $Z$ and is a $C^{\infty}$ map. Note that since $a$ is smooth and since $C^s(M)$ is an algebra under pointwise multiplication, the map 
	\[
	u \mapsto a(x,u)
	\]
	takes $C^s(M)$ to $C^s(M)$, and if $\norm{u}_{C^s(M)} \leq K$ then $\norm{a(x,u)}_{C^s(M)} \leq C(a,s,K)$ (these facts follow from a local coordinate computation). Thus $F$ is well defined. If $u, v \in C^s(M)$ we use the Taylor formula  
	\[
	a(x,u+v) = \sum_{j=0}^m \frac{\p_u^j a(x,u)}{j!} v^j + \int_0^1 \frac{\p_u^{m+1} a(x,u+tv)}{m!} v^{m+1} (1-t)^m \,dt.
	\]
	Since $C^s(M)$ is an algebra, we have that when $\norm{v}_{C^s(M)} \leq 1$ one has 
	\[
	\left \| a(x,u+v) - \sum_{j=0}^m \frac{\p_u^j a(x,u)}{j!} v^j\right\|_{C^s(M)} \leq C_{m,a,u} \norm{v}_{C^s(M)}^{m+1}.
	\]
	This shows that $u \mapsto a(x,u)$ is a $C^{\infty}$ map $C^s(M) \to C^s(M)$. Since the other parts of $F$ are linear, $F$ is a $C^{\infty}$ map in the standard sense of~\cite[Definition 10.2]{renardy2006introduction}.
    
    Note that $F(0,0) = 0$ by \eqref{a_first_condition}. The linearization of $F$ at $(0,0)$ in the $u$-variable is 
    \[
    \left. D_uF\right|_{(0,0)}(v) = (\Delta_g v + \p_u a(x,0) v,v|_{\p M}).
    \]
    This is a homeomorphism $Y\to Z$ by \eqref{a_second_condition}. To see this, let $(w,\phi)\in Z=C^{s-2}(M)\times C^s(\p M)$, and consider the Dirichlet problem 
    \begin{align*}
    \begin{cases}
    (\Delta_g + \p_u a(x,0))v=w & \text{ in } M, \\
    v=\phi & \text{ on } \p M.
    \end{cases}
    \end{align*}
    Solutions are unique by \eqref{a_second_condition}, and using the Fredholm alternative and Schauder estimates there is a unique solution $v\in Y=C^s(M)$ (see e.g.\ \cite[Exercise 1 in Section 13.8]{taylor2011partial}). 
    Thus the implicit function theorem in Banach spaces~\cite[Theorem 10.6 and Remark 10.5]{renardy2006introduction} yields that there is $\delta>0$ and an open ball $U_{\delta}=B_X(0,\delta)\subset X$ and a $C^{\infty}$ map $S: U_{\delta} \to Y$ such that whenever $\norm{f}_{C^s(\p M)}\leq \delta$ we have 
    \[
     F(f,S(f))=(0,0).
    \]
    Since $S$ is Lipschitz continuous and $S(0) = 0$, $u = S(f)$ satisfies 
    \[
     \norm{u}_{C^{s}(M)}\leq C \norm{f}_{C^s(\p M)}.
    \]
    Moreover, by redefining $\delta$ if necessary $u=S(f)$ is the only solution to $F(f,u)=(0,0)$ whenever $\norm{f}_{C^s(\p M)}\leq \delta$ and $\norm{u}_{C^{s}(M)}\leq C \delta$.
	We have proven the existence of unique small solutions of the Dirichlet problem~\eqref{solvability} and the fact that the solution operator $S: U_{\delta} \to C^{s}(M)$ is a $C^{\infty}$ map. Since the normal derivative is a linear map $C^{s}(M) \to C^{s-1}(\p M)$, it follows that also $\Lambda$ is a well defined $C^{\infty}$ map $U_{\delta} \to C^{s-1}(\p M)$.
\end{proof}	

We now specialize to a power type nonlinearity, for which the higher order linearizations of the DN map will be particularly simple. The next proposition justifies the formal calculation that we may differentiate the equation 
\begin{equation}\label{formal_calc_eq}
\Delta_g u + q(x) u^m = 0 \text{ in $M$}, \quad u_f|_{\p M} = \eps_1f_1+\cdots+\eps_m f_m
\end{equation}
in the $\eps_j$ variables to have equations corresponding to first and $m$th linearizations, 
\[
 \Delta_gv_{f_k}=0 \text{ and }  \Delta_gw=-(m!)qv_{f_1}\cdots v_{f_m}.
\]
The normal derivative of $w$ is the $m$th linearization of the DN map of~\eqref{formal_calc_eq}. In the proposition, we write 
\[
(D^k f)_x(y_1,\ldots,y_k) 
\]
to denote the $k$th derivative at $x$ of a mapping $f$ between Banach spaces, considered as a symmetric $k$-linear form acting on $(y_1,\ldots,y_k)$. We refer to  \cite[Section 1.1]{hormander1983analysis}, where the notation $f^{(k)}(x; y_1, \ldots, y_k)$ is used instead of $(D^k f)_x(y_1,\ldots,y_k)$. 

\begin{prop}\label{derivs_and_integral_formula}
Let $q \in C^{\infty}(M)$, and let $\Lambda_q$ be the DN map for the semilinear equation 
\begin{equation}\label{Dir_prblm}
\Delta_g u + q(x) u^m = 0 \text{ in $M$},
\end{equation}
where
\[
 m\in \N \text{ and } m\geq 2.
\]
For $f \in C^s(\p M)$, let $v_f$ be the solution of the Laplace equation 
	\begin{equation}\label{first_lin}
	\Delta_g v_f = 0 \text{ in $M$}, \qquad v_f|_{\p M} = f.
	\end{equation}
The first linearization $(D\Lambda_q)_0$ of $\Lambda_q$ at $f=0$ is the DN map of the Laplace equation:
\[
(D\Lambda_q)_0: C^s(\p M) \to C^{s-1}(\p M), \ \ f \mapsto \p_{\nu} v_f|_{\p M}.
\]
The higher order linearizations $(D^j \Lambda_q)_0$ are identically zero for $2 \leq j \leq m-1$. 

The $m$-th linearization $(D^m \Lambda_q)_0$ of $\Lambda_q$ at $f=0$ is characterized by the following identity:  
	for any $f_1, \ldots, f_{m+1} \in C^{s}(\p M)$ one has 
	\begin{equation} \label{dm_lambdaq_identity}
	\int_{\p M} (D^m \Lambda_q)_0(f_1, \ldots, f_m)f_{m+1}  \,dS = -(m!) \int_M  q v_{f_1} \cdots v_{f_{m+1}} \,dV
	\end{equation}
	here each $v_{f_k}$, $k=1,\ldots,m+1$, is a solution to~\eqref{first_lin} with boundary value $f=f_k$.
\end{prop}
\begin{proof}
The nonlinearity $a(x,u) = q(x) u^m$ satisfies the conditions in Proposition \ref{wellposedness_and_expansion}, and thus the DN map $\Lambda_q = \p_{\nu} S|_{\p M}$ is well defined for small data. Here $S: f \mapsto u_f$ is the solution operator for the Dirichlet problem of the equation~\eqref{Dir_prblm}. To compute the derivatives of $\Lambda_q$ at $0$, it is enough to consider derivatives of $S$. Writing $f := \eps_1 f_1 + \ldots + \eps_k f_k$, the function $u_f \in C^{s}(M)$ depends smoothly on $\eps_1, \ldots, \eps_k$ since $S$ is smooth. Applying $\p_{\eps_1} \cdots \p_{\eps_k}|_{\eps_1=\ldots=\eps_k=0}$ to the Taylor formula (see e.g.\ \cite[Section 1.1]{hormander1983analysis})
\[
S(f) = \sum_{j=0}^k \frac{(D^j S)_0(f, \ldots, f)}{j!} + \int_0^1 \frac{(D^{k+1} S)_{tf}(f, \ldots, f)}{k!} (1-t)^k \,dt
\]
implies that $(D^k S)_0$ may be computed using the formula 
\[
(D^k S)_0(f_1, \ldots, f_k) = \p_{\eps_1} \cdots \p_{\eps_k} u_f|_{\eps_1=\ldots=\eps_k=0}.
\]
Moreover, since $u_f$ is smooth in the $\eps_j$ variables and $\Delta_g$ is linear, we may differentiate the equation 
\begin{equation} \label{mth_power_nonlinearity}
\Delta_g u_f + q(x) u_f^m = 0, \qquad u_f|_{\p M} = f
\end{equation}
freely in the $\eps_j$ variables.

Let first $k = 1$, so that $u = u_{\eps_1 f_1}$. Since $u_0 = 0$ and $m \geq 2$, the derivative of \eqref{mth_power_nonlinearity} in $\eps_1$ evaluated at $\eps_1 = 0$ satisfies 
\[
\Delta_g (\p_{\eps_1} u_f|_{\eps_1 = 0}) = 0, \qquad \p_{\eps_1} u_f|_{\p M} = f_1.
\]
Thus the first linearization of the map $S$ at $f=0$ is 
\[
(DS)_0(f_1) = \p_{\eps_1} u_{\eps_1 f_1}|_{\eps_1=0} = v_{f_1},
\]
where $v_{f_1}$ satisfies~\eqref{first_lin} with $f=f_1$.

For $2 \leq k \leq m-1$, applying $\p_{\eps_1} \cdots \p_{\eps_k}|_{\eps_1=\ldots=\eps_k=0}$ to \eqref{mth_power_nonlinearity} gives that 
\[
\Delta_g (\p_{\eps_1} \cdots \p_{\eps_k} u_f|_{\eps_1=\ldots=\eps_k=0}) = 0, \qquad \p_{\eps_1} \cdots \p_{\eps_k} u_f|_{\p M} = 0,
\]
since $\p_{\eps_1} \cdots \p_{\eps_k} (q(x) u_f^m)$ is a sum of terms containing positive powers of $u_f$, which are equal to zero when $f = 0$. Uniqueness of solutions for the Laplace equation implies that 
\[
(D^j S)_0(f_1, \ldots, f_k) = 0, \qquad 2 \leq k \leq m-1.
\]
When $k=m$, the only term in $\p_{\eps_1} \cdots \p_{\eps_m} (q(x) u_f^m)$ which does not contain a positive power of $u_f$ is $q(x) (m!) (\p_{\eps_1} u_f) \cdots (\p_{\eps_m} u_f)$. This is the only nonzero term after setting $\eps_1 = \ldots = \eps_m = 0$, and thus the function 
\[
w:= (D^m S)_0(f_1, \ldots, f_m) = \p_{\eps_1} \cdots \p_{\eps_m} u_f|_{\eps_1=\ldots=\eps_m=0}
\]
solves 
\begin{equation} \label{laplace_um_equation}
\Delta_g w = - q(x) (m!) v_{f_1} \cdots v_{f_m} \text{ in $M$}
\end{equation}
with zero Dirichlet boundary values.

By linearity one has 
\[
(D^k \Lambda_q)_0 = \p_{\nu} (D^k S)_0|_{\p M}.
\]
The claims for $(D^k \Lambda_q)_0$ when $1 \leq k \leq m-1$ follow immediately. For $k = m$ we observe that  $(D^m \Lambda_q)_0(f_1,\ldots,f_m) = \p_{\nu} w|_{\p M}$ satisfies 
\[
\int_{\p M} (\p_{\nu} w) f_{m+1} \,dS = \int_M ( \langle d w, d v_{f_{m+1}} \rangle_g + (\Delta_g w) v_{f_{m+1}} ) \,dV.
\]
The integral of $\langle d w, d v_{f_{m+1}} \rangle_g$ vanishes since $w|_{\p M} = 0$ and $v_{f_{m+1}}$ is harmonic. The proposition follows by using \eqref{laplace_um_equation}.
\end{proof}

\section{Proof of Theorem \ref{main thm q}}\label{Section 3}

In this section, we use the higher order linearization approach (in fact, the second order linearization of the DN map) to prove Theorem \ref{main thm q}. 
\begin{proof}[Proof of Theorem \ref{main thm q}]
	We could use Proposition~\ref{derivs_and_integral_formula} to have the integral equation~\eqref{integral id in Rn} below directly, even for the product of three harmonic functions instead of two (this is a stronger statement since one can always take the third harmonic function to be constant). The theorem would follow from this by using harmonic exponentials. However, we choose to give a direct hands-on approach that describes how to use the method. 
	
	Let $\eps_1, \eps_2$ be sufficiently small numbers and let $f_1,f_2\in C^\infty(\p M)$. 
	Let the function $u_j:= u_j (x;\eps_1,\eps_2)\in C^s(M)$ be the unique small solution of 
	\begin{align}\label{equ1 in 1st example}
	\begin{cases}
	\Delta u_j +q_j u_j^2 =0 &\text{ in }\Omega,\\
	u_j =\eps_1 f_1 +\eps_2 f_2 &\text{ on }\p \Omega,
	\end{cases}
	\end{align}
	for $j=1,2$ provided by Proposition~\ref{derivs_and_integral_formula}. Let us differentiate \eqref{equ1 in 1st example} with respect to $\epsilon_{\ell}$  so that 
	\begin{align}\label{equ2 in 1st example}
	\begin{cases}
	\Delta \left(\frac{\p}{\p \eps_\ell}u_j\right) +2 q_j u_j \left(\frac{\p}{\p \eps_\ell}u_j\right)=0 & \text{ in }\Omega, \\
	\frac{\p}{\p \eps_\ell}u_j = f_\ell &\text{ on }\p \Omega.
	\end{cases}
	\end{align}
	Inserting $\eps_1=\eps_2=0$ into \eqref{equ2 in 1st example}, shows that 
	\begin{align*}\label{equ3 in 1st example}
	\Delta v_j^{(\ell)} =0 \text{ in }\Omega \text{ with }v_j^{(\ell)}=f_\ell \text{ on }\p \Omega,
	\end{align*}
	where 
	\[
	v_j^{(\ell)}(x) =\left.\frac{\p}{\p \eps_\ell}\Big|_{\eps_1=\eps_2=0}u_j(x;\eps_1,\eps_2)\right..
	\]
	Here we used  $u_j(x;0,0)\equiv 0$. 
	The functions $v_j^\ell$ are just harmonic functions defined in $\Omega$ with boundary data $f_\ell|_{\p \Omega}$.
	By uniqueness of the Dirichlet problem for the Laplacian we have that 
	\begin{align}\label{v_1 =v_2 Rn}
	v^{(\ell)}:=v_1^{(\ell)}=v_2^{(\ell)} \text{ in } \Omega  \text{ for } \ell=1,2.
	\end{align} 
	
	Next, let us differentiate \eqref{equ2 in 1st example} with respect to $\epsilon_k$ for $k\neq \ell$. Then we have that
	\begin{align}\label{equ 3 in 1st example}
	\begin{cases}
	\Delta \left(\frac{\p^2}{\p \eps_1 \p \eps_2}u_j \right) + 2q_j u_j \left(\frac{\p^2}{\p \eps_1 \p \eps_2}u_j \right) + 2 q_j \left(\frac{\p u_j}{\p \eps_1}\right) \left(\frac{\p u_j }{\p \eps_2} \right) =0 & \text{ in }\Omega, \\
	\frac{\p^2}{\p \eps_1 \p \eps_2}u_j =0 & \text{ on }\p \Omega.
	\end{cases}
	\end{align}
	Again, evaluating at $\eps_1=\eps_2=0$, the equation \eqref{equ 3 in 1st example} becomes
	\begin{align}\label{equ 4 in 1st example}
	\begin{cases}
	\Delta  w_j + 2 q_j v^{(1)}v^{(2)}  =0 & \text{ in }\Omega, \\
	w_j=0 & \text{ on }\p \Omega,
	\end{cases}
	\end{align}
	where $w_j(x)=\left(\frac{\p^2}{\p \eps_1 \p \eps_2}u_j\right)(x;0,0)$ and we used $u_j(x;0,0)\equiv 0$ for $j=1,2$ again. By using the fact that $\Lambda_{q_1}(\eps_1 f_1 +\eps_2 f_2)=\Lambda_{q_2}(\eps_1 f_1 +\eps_2 f_2)$ for small $\eps_1, \eps_2$, we have 
	\[
	\p_{\nu} u_1|_{\p \Omega} = \p_{\nu} u_2|_{\p \Omega},
	\]
	and applying $\p_{\eps_1} \p_{\eps_2}|_{\eps_1=\eps_2=0}$ to this identity gives that 
	\[
	\p_\nu w_1|_{\p \Omega}=\p_\nu w_2|_{\p \Omega}. 
	\]
	Thus, by integrating the equation~\eqref{equ 4 in 1st example} over $\Omega$ (i.e.\ integrating against the harmonic function $v^{(3)} = 1$) and by using integration by parts we have	
	\begin{align}\label{integral id in Rn}
	0=\int_{\p \Omega} (\p_\nu w_1-\p_\nu w_2) \,dS=\int_{\Omega} \Delta(w_1 -w_2) \,dx=2\int_\Omega (q_2-q_1)v^{(1)}v^{(2)} \,dx
	\end{align}
	where $v^{(1)}$ and $v^{(2)}$ are defined in \eqref{v_1 =v_2 Rn}. Therefore, by choosing $f_1$ and $f_2$ as the boundary values of the Calder\'on's exponential solutions~\cite{calderon2006inverse},
	\begin{align}\label{exp_sols}
	v^{(1)}(x):=\exp((k+i\xi)\cdot x), \quad v^{(2)}(x):=\exp((-k+i\xi)\cdot x),
	\end{align}
	where $k,\xi\in \R^n$, $k\perp \xi$ and $\abs{k}=\abs{\xi}$, we obtain that the Fourier transformation of the difference $q_2-q_1$ is zero. Thus $q_1=q_2$. 
\end{proof}
In the proof above we did not need to construct special solutions for an elliptic equation with unknown coefficients, such as complex geometrical optics solutions. The linearization technique allowed us to simply use known harmonic functions. 
This fact gives an extremely simple reconstruction in the setting of Theorem~\ref{main thm q}. 
\begin{cor} \label{cor main thm q}
	Let $n\geq 2$, and let $\Omega\subset \R^n$ be a bounded domain with $C^\infty$ boundary $\p \Omega$. Assume that $q\in C^\infty(\ol{\Omega})$, and let $\Lambda_q$ be the DN map for the equation 
	\begin{equation*}\label{simplest_nonlinear_reconstruct}
	\Delta u + qu^2 =0 \text{ in } \Omega.
	\end{equation*}
	Then
	\begin{align}\label{reconstruction formula}
	\widehat{q}(-2 \xi)=-\frac{1}{2}\int_{\p \Omega}\frac{\p^2}{\p \eps_1 \p \eps_2}\Big|_{\eps_1=\eps_2=0}\Lambda_q(\eps_1f_1+\eps_2f_2) \,dS,
	\end{align}
	where $f_1$ and $f_2$ are the boundary values of the exponential solutions~\eqref{exp_sols} and $\widehat{q}$ stands for the Fourier transform of $q$.
\end{cor}
\begin{proof}
	The proof of the reconstruction can be directly read from \eqref{integral id in Rn} in the proof of Theorem~\ref{main thm q}.
\end{proof}

We end of this section with a remark about the stability of the reconstruction formula in Corollary~\ref{cor main thm q}.
\begin{rmk}
	From the above reconstruction formula, we adapt the same notations in Section \ref{Section 2}, then we can simply rewrite the formula \eqref{reconstruction formula} as 
	\[
	\widehat{q}_j(-2\xi) =-\frac{1}{2} \int_{\p \Omega}(D^2 \Lambda_{q_j})_0(f_1, f_2) \,dS, \text{ for }j=1,2.
	\]
	Subtracting these two formulas gives that 
	\[
	(\widehat{q}_1-\widehat{q}_2)(-2\xi) = -\frac{1}{2} \int_{\p \Omega} ((D^2 \Lambda_{q_1})_0 - (D^2 \Lambda_{q_2})_0)(f_1, f_2) \,dS.
	\]
Now, we assume that $\norm{D^k(\Lambda_{q_1}-\Lambda_{q_2})_0}_*$ is sufficiently small for $k=0,1,2$, and $\norm{q_j}_{H^1(\Omega)} \leq R$ for $j=1,2$, where 
	\[
	\norm{T}_*=\sup_{\norm{f_1}_{C^s(\p \Omega)}=\ldots=\norm{f_k}_{C^s(\p \Omega)}=1} \norm{T(f_1, \ldots, f_k)}_{C^{s-1}(\p \Omega)}
	\]
	for a bounded $k$-linear form $T: C^s(\p \Omega) \times \ldots \times C^s(\p \Omega) \to C^{s-1}(\p \Omega)$. Next, by taking harmonic functions $v_{f_1}=v^{(1)}$, $v_{f_2}=v^{(2)}$in $\Omega$, where $v^{(1)},v^{(2)}$ are the functions defined in \eqref{exp_sols}, one can obtain that 
	\begin{align}\label{log stability}
	\norm{q_1-q_2}_{L^2} \leq \omega \left(\norm{D^2(\Lambda_{q_1}-\Lambda_{q_2})_0}_*\right),
	\end{align}
	where $\omega(t)$ is a modulus of continuity satisfying, for some $C = C(R)$, 
	\[
	\omega(t) \leq C| \log t|^{-\frac{2}{n+2}}, \quad 0<t<\frac{1}{e}.
	\]
	One can directly prove the logarithmic stability \eqref{log stability} by using standard arguments in stability for the Calder\'on problem, for example, see \cite[Section 4]{salo2008calderon}.
\end{rmk}

\section{Simultaneous recovery on two-dimensional Riemannian surfaces}\label{Section 4}

We use the higher order linearization approach to simultaneously recover, from the DN map, the conformal class of a Riemannian surface and the potential of a semilinear Schr\"odinger operator up to the gauge symmetry. 
We use first order linearization to recover the conformal class of the manifold by using the result~\cite{lassas2001determining} (see also \cite{lassas2018poisson} for a recent alternative proof). Then by using the result~\cite{guillarmou2011calderon} we recover the potential on the known conformal manifold (up to gauge).

\begin{proof}[Proof of Theorem \ref{main thm 2D}]
The proof is divided into two steps. We first recover the manifold and the conformal class of the metric. After that we recover the potential on a known manifold up to the gauge. 

	\vspace{10pt}
	
{\it Step 1. Recovering the conformal manifold.}\\

\noindent  Note first that 
\begin{align*}
&\Lambda_{q_1}(f) = \Lambda_{q_2}(f) \text{ for small $f$} \\
\implies & (D\Lambda_{q_1})_0 = (D\Lambda_{q_2})_0.
\end{align*}
By Proposition~\ref{derivs_and_integral_formula}, $(D\Lambda_{q_j})_0$ is the DN map for the Laplace equation on $(M,g_j)$, i.e.\ we have that the DN maps of the following Dirichlet problems
\begin{align*}
\begin{cases}
\Delta_{g_j} v_j =0 & \text{ in }M_j, \\
v_j=f & \text{ on } \p M
\end{cases}
	\end{align*}
	agree.
%
Thus by using
\cite[Theorem 5.1]{lassas2018poisson}, one can determine the manifold and the Riemannian metric up to a conformal transformation. That is, there exists a conformal $C^\infty$ diffeomorphism $J$ such that $\sigma J^*g_2= g_1$ with $J|_{\p M}=\Id$, for some smooth positive function $\sigma\in C^\infty(M_1)$ with $\sigma |_{\p M} =1$. This completes the first part of the proof.
 
	\vspace{10pt}
	
{\it Step 2. Recovering the potential.}\\

\noindent 
Let us make a change of coordinates to pass from the equation~\eqref{2Deq} on $(M_2,g_2)$ onto the manifold $(M_1,g_1)$. We denote
\[
\widetilde{q}_2=\sigma^{-1}q_2\circ J \equiv \sigma^{-1}J^*q_2.
\]
Let $f\in  C^s(\p M)$ be small and let $u_2$ be the solution to
\[
\Delta_{g_2}u_2+q_2u_2^m=0 \text{ in }M_2\text{ with } u_2=f \text { on } \p M.
\]
Let us denote
\[
\widetilde{u}_2:=J^*u_2\equiv u_2\circ J.
\]
Then  $\widetilde{u}_2$ solves
\begin{align*}
\Delta_{g_1}\widetilde{u}_2+\widetilde{q}_2(\widetilde{u}_2)^m&=\Delta_{\sigma J^*g_2}\widetilde{u}_2 +\widetilde{q}_2(\widetilde{u}_2)^m \\
&=\sigma ^{-1}\Delta_{J^*g_2}\widetilde{u}_2 +\sigma^{-1}(J^*q_2)(\widetilde{u}_2)^m \\
&= \sigma ^{-1}J^*(\Delta_{g_2}u_2)+ \sigma^{-1}(J^*q_2)(J^*u_2)^m \\
&=\sigma^{-1}\left[J^*(\Delta_{g_2} u_2 ) + (J^*q_2)(J^* u_2)^m\right].
\end{align*}
Here we used the conformal invariance of the Laplace-Beltrami operator in dimension $2$ in the second equality. In the third equality, the coordinate invariance of Laplace-Beltrami operator was used.
Since $u_2$ solves $\Delta_{g_2}u_2+ q_2u_2 ^m=0$ in $M_2$, we trivially have that $J^*(\Delta_{g_2} u_2 ) + (q_2 \circ J)(J^* u_2)^m=0$ in $M_1$. 
Consequently, we have that 
\begin{align}\label{conformal change}
	\begin{cases}
	\Delta_{g_1}\widetilde{u}_2+\widetilde{q}_2(\widetilde{u}_2)^m=0 & \text{ in }M_1,\\
	\widetilde{u}_2=f & \text{ on }\p M,
	\end{cases}
\end{align}
since $\sigma|_{\p M}=1$ and $J|_{\p M}=\text{Id}$.

Next, let $u_1$ solve the nonlinear equation~\eqref{2Deq} on $(M_1,g_1)$ with potential $q_1$ and boundary value $f$. We show that there holds
\begin{equation}\label{DN_maps_on_single_manifold}
\p_{\nu_1} u_1=\p_{\nu_1}\widetilde{u}_2 \text{ on }\p M,
\end{equation}
by the assumption that $\Lambda_{M_1,g_1,q_1}=\Lambda_{M_2,g_2,q_2}$ for small data. 
Since $\Lambda_{M_1,g_1,q_1}=\Lambda_{M_2,g_2,q_2}$, it follows that if $u_1 =u_2=f$  on  $\p M$, then 
\begin{equation*}\label{DNmapsgive}
\p_{\nu_1}u_1=\p_{\nu_2}u_2 \text{ on }\p M.
\end{equation*}
We calculate
\[
\p_{\nu_2}u_2=\nu_2\cdot du_2=\nu_2\cdot d(u_2\circ J \circ J^{-1})=(J^{-1}_*\nu_2)\cdot d\tilde{u}_2=\nu_1\cdot d\tilde{u}_2=\p_{\nu_1} \widetilde{u}_2.
%
\]
Here $\cdot$ denotes the canonical pairing between vectors $\nu_2\cdot du_2=\nu_2^k \p_k u$ (with Einstein summation implied), and we used the facts that $J$ is diffeomorphic conformal mapping with $J|_{\p M}=\text{Id}$ and $\sigma|_{\p M} =1$. Thus we have~\eqref{DN_maps_on_single_manifold} and consequently
\begin{equation*}\label{transfd_DN_maps_agree}
\Lambda_{M_1,g_1,q_1}(f) = \widetilde{\Lambda}_{M_1,g_1,\tilde{q}_2}(f) \text{ for small $f$},
\end{equation*}
where $\widetilde{\Lambda}_{M_1,g_1,\tilde{q}_2}$ stands for the DN map of the Dirichlet problem \eqref{conformal change}.

By Proposition~\ref{derivs_and_integral_formula}, we have 
\begin{align*}
 &\Lambda_{M_1,g_1,q_1}(f) = \widetilde{\Lambda}_{M_1,g_1,\tilde{q}_2}(f) \text{ for small $f$} \\
 \implies & (D^2 \Lambda_{M_1,g_1,q_1})_0 = (D^2 \widetilde{\Lambda}_{M_1,g_1,\tilde{q}_2})_0 \\
 \implies & \int_{M_1} (q_1-\widetilde{q}_2) v_1 v_2 v_3 \,dV = 0
\end{align*}
where $v_j \in C^s(M_1)$ are harmonic functions in $M_1$. Choosing $v_3 = 1$, we get 
\[
\int_{M_1} (q_1-\widetilde{q}_2) v_1 v_2 \,dV = 0
\]
for any harmonic functions $v_1$ and $v_2$ in $M_1$. Choosing $v_j$ to be complex geometrical optics solutions constructed in \cite{guillarmou2011calderon} (see the proof of Proposition 5.1 in that paper, actually since $v_j$ are harmonic Carleman estimates are not needed and the construction in \cite{guillarmousalotzou_complex} would suffice), it follows that 
\[
 q_1 = \widetilde{q}_2 \text{ in }M_1.
\]
This concludes the proof. 
\end{proof}


\section{Global uniqueness on transversally anisotropic manifolds}\label{Section 5}

We will prove the following result, whose proof is based on the construction in \cite{ferreira2013calderon} of harmonic functions on transversally anisotropic manifolds that concentrate near certain two dimensional surfaces.

\begin{prop}\label{Four_Gaussian_beams}
	Let $(M,g)$ be a transversally anisotropic manifold and assume that $m \geq 4$. If $f \in C^1(M)$ satisfies 
	\begin{align}\label{integral id to zero on CTA}
	\int_M f u_1 \cdots u_m \,dV = 0
	\end{align}
	for any $u_j \in C^{\infty}(M)$ with $\Delta_g u_j = 0$ in $M$, then $f \equiv 0$.
\end{prop}

Theorem \ref{thm for CTA} follows immediately from this proposition:

\begin{proof}[Proof of Theorem \ref{thm for CTA}]
	Let $\Lambda_{q_j}$ be the DN map for the equation $\Delta_g u + q u^m = 0$ in $M$. If $\Lambda_{q_1}(f) = \Lambda_{q_2}(f)$ for small $f$, then $(D^m \Lambda_{q_1})_0 = (D^m \Lambda_{q_2})_0$ and thus by Proposition \ref{derivs_and_integral_formula} one has 
	\[
	\int_M (q_1-q_2) v_1 \cdots v_{m+1} \,dV = 0
	\]
	where $v_j \in C^s(M)$ are harmonic functions in $M$. Since $m \geq 3$, it follows from Proposition \ref{Four_Gaussian_beams} that $q_1 = q_2$.
\end{proof}

Before we prove Proposition \ref{Four_Gaussian_beams} in the  general case, let us explain the main idea of the proof in a simplified setting. Consider the situation where each point $y_0 \in M_0$ on the transversal manifold $M_0$ is an intersection point of two distinct nontangential geodesics $\gamma,\eta$ (depending on $y_0 $), which have no other intersection points. Assume also that $\gamma$ and $\eta$ do not self-intersect. 

As in \cite[Section 2]{ferreira2013calderon}, we choose special harmonic functions $u_1,u_2,u_3$ and $u_4$ of the form
\begin{align*}
u_1 = e^{-(\tau+i\lambda) x_1} ( \widetilde{v}_{\tau+i\lambda}+ r_1), \qquad 
&u_2 = \overline{e^{(\tau+i\lambda) x_1} ( \widetilde{v}_{\tau+i\lambda}+ r_2)}, \\
u_3 = e^{-\tau x_1} ( \widetilde{w}_{\tau}+ r_3),  \qquad 
&u_4 = \overline{e^{\tau x_1} ( \widetilde{w}_{\tau}+ r_4)}.
\end{align*}
Here $\widetilde{v}_{\tau+i\lambda}$, $\widetilde{w}_{\tau}$ are Gaussian beam quasimodes on $M_0$ concentrating near $\gamma$ and $\eta$, respectively, as in \cite[Proposition 3.1]{ferreira2013calderon}, and the functions $r_\ell$, $1\leq \ell \leq 4$, are small correction terms. Here $x_1$ is the variable of the Euclidean direction.

Since $\widetilde{v}_{\tau+i\lambda}$ and $\widetilde{w}_{\tau}$ are supported near $\gamma$ and $\eta$, respectively, and since $\gamma$ and $\eta$ only intersect at $y_0$, the product $u_1 u_2 u_3 u_4$ concentrates near the point $y_0$. Thus by using these solutions $u_\ell$, $1\leq \ell\leq 4$, it follows by the assumption \eqref{integral id to zero on CTA} that 
\begin{align}\label{CTA some equa}
\int_{B_\delta(y_0)} \hat{f}(2\lambda, \cdot) e^{2i\tau \mathrm{Im}(\Phi+\Psi )} A \,dV_{g_0}=\mathcal{O}\left(\tau^{-R} \right),
\end{align}
where $R>0$ can be chosen arbitrarily large, $B_\delta(y_0)$ is a geodesic ball of radius $\delta>0$ in $M_0$, and $\hat{f}$ denotes the partial Fourier transform of $f$ with respect to $x_1$. Here $\Phi$ and $\Psi$ are the phase functions of the Gaussian beams corresponding to $\gamma$ and $\eta$, respectively, and $A$ is an amplitude function, which is nonzero at the intersection point $y_0\in M_0$.

The point here is the following. The Hessians of $\mathrm{Im}(\Phi)$ and $\mathrm{Im}(\Psi)$ at $y_0$ are positive definite in directions orthogonal to $\gamma$ and $\eta$, respectively. In fact, the Hessians of $\mathrm{Im}(\Phi)$ and $\mathrm{Im}(\Psi)$ are also nonnegative definite and their gradients vanish on $\gamma$ and $\eta$, respectively. Since $\gamma $ and $\eta $ intersect at $y_0 $, the sum $\mathrm{Im}(\Phi+\Psi)$ is positive definite at $y_0 $ and $d(\Phi+\Psi)(y_0 )=0$. By multiplying \eqref{CTA some equa} by $\tau^{1/2}$, it follows by taking $\tau \to \infty$ that 
\[
c\hat{f}(2\lambda,y_0 )=0,
\]
for some constant $c\neq 0$. Since $\lambda\in \R$ is arbitrary and $y_0 \in M_0$ is also arbitrary, this proves the proposition in this special case. For details we refer to the following proof.

To deal with a transversally anisotropic manifold, we need to consider the (in most cases rare) possibility where $\gamma$ and $\eta$ may intersect at many different points and may have self-intersections. This makes the proof much more technical, and will be achieved by introducing additional parameters in the above construction and by using the auxiliary Lemma \ref{lemma_linear_combination_analytic} below.

Now, we turn to prove Proposition \ref{Four_Gaussian_beams} in the general case.

\begin{proof}[Proof of Proposition \ref{Four_Gaussian_beams}]

	We do the proof in several steps.
	
	\vspace{10pt}

{\it Step 1. Preparation.} \\

\noindent 
	By taking $u_j = 1$ for $j \geq 5$, it is sufficient to prove the result when $m=4$. 
	We assume that $M \subset \subset \R \times M_0$ with $g = e \oplus g_0$. 
	The dimension of $M$ is denoted by $n$, and so $\dim(M_0)=n-1$.
	
	We may enlarge $M_0$ so that it has strictly convex boundary (first embed $M_0$ in some closed manifold $M_1$, remove a small geodesic ball from $M_1 \setminus M_0$ and glue a part with strictly convex boundary near the removed part). By \cite[Lemma 3.1]{salo2017normal}, there is a set $E$ of zero measure in $(M_0,g_0)$ so that all points in $M_0 \setminus E$ lie on some nontangential geodesic between boundary points. Fix a point $y_0 \in M_0 \setminus E$ and a direction $v_0 \in S_{y_0} M_0$ so that the geodesic $\gamma: [0,T] \to M_0$ through $(y_0,v_0)$ is a nontangential geodesic between boundary points. Without loss of generality, we may assume that $\gamma$ is a unit speed geodesic (i.e. $\abs{\dot{\gamma}}=1$). The property of a geodesic being nontangential is not changed under small perturbations. Therefore, we may find $w_0 \in S_{y_0} M_0$ close to $v_0$ so that $w_0 \neq v_0$ and the unit speed geodesic $\eta: [0,S] \to M_0$ through $(y_0, w_0)$ is also a nontangential geodesic between boundary points. We may arrange so that the geodesics $\gamma$ and $\eta$ are distinct and are not reverses of each other (in fact, $\gamma$ can only self-intersect at $y_0$ finitely many times \cite[Lemma 7.2]{kenigsalo_partial}, and it is enough choose $w_0$ near $v_0$ that is different from the corresponding finitely many tangent vectors of $\gamma$ and their negatives).

	We next show that two distinct geodesics $\gamma$ and $\eta$ that are not the reverse of each other can intersect only finitely many times.
	Assume the opposite, that there are infinitely many intersection points $\{p_k \}_{k\in \N}$ and intersection times $\{t_k\}_{k\in \N}$, $\{s_k\}_{k\in \N}$ satisfying 
	 \[
     \gamma(t_k)=p_k =\eta(s_k), \text{ for all }k\in \N.
	 \]
	Since $M$ is compact, $t_k\in [0,T]$ and $s_k\in [0,S]$, by passing to subsequences and using continuity of unit speed geodesics $\gamma,\eta$, we may assume that $\gamma(t_k)\to \gamma(t_0)= p_0$ and $\eta(s_k)\to \eta(s_0)= p_0$ as $k\to \infty$, for some $t_0 \in [0,T]$, $s_0 \in [0,S]$ and $p_0\in M$.
	
	In addition, we denote the tangent vectors $V_\gamma:=\dot{\gamma}(t_0)$ and $V_\eta:=\dot{\eta}(s_0)$. By using the continuity of $\dot{\gamma}(t)$, $\dot{\eta}(s)$ and the compactness of the unit sphere, we have (by passing to subsequences again) that 
	\[
	\lim_{k\to \infty}\dot{\gamma}(t_k)=V_\gamma \text{ and }\lim_{k\to \infty}\dot{\eta}(s_k)=V_\eta.
	\]
	Now, it is clear that $V_\gamma\neq \pm V_\eta$, by using the fact that $\gamma$ and $\eta$ are distinct and not reverses. The injectivity radius at $p_0$ is positive. However, since $\gamma$ and $\eta$ intersect in all geodesic balls $B_\eps(p_0)$ for any $\eps>0$, this is a contradiction. This shows that two different nontangential geodesics can only intersect finitely many times.

		\vspace{10pt}

{\it Step 2. Choice of the harmonic functions $u_j$.} \\

\noindent We denote the points of $\R\oplus M_0$ by $(x_1,x')$.
	We now use the argument in \cite[Proposition 2.2]{ferreira2013calderon} and choose solutions 
	\begin{align*}
		u_1 = e^{-L(\tau+i\lambda) x_1} ( \widetilde{v}_{L(\tau+i\lambda)}(x') + r_1), \qquad 
		u_2 = \overline{e^{L(\tau+i\lambda) x_1} ( \widetilde{v}_{L(\tau+i\lambda) }(x') + r_2)},
	\end{align*}
	of $\Delta_g u_j = 0$ in $M$, where $\tau \geq 1$ is sufficiently large, $L\geq 1$ is an additional large parameter that will be fixed later, and $\lambda \in \C$ is fixed, 
	\[
	\widetilde{v}_{L(\tau+i\lambda) } := \tau^{-\frac{n-2}{8}} v_{L(\tau+i\lambda) }(x')
	\]
	where $v_{L(\tau+i\lambda) }$ is the Gaussian beam quasimode concentrating near the geodesic $\gamma$ in $M_0$ constructed in \cite[Proposition 3.1]{ferreira2013calderon}, and $r_j$ are remainder terms satisfying 
	\[
	\norm{r_j}_{H^k(M)} = O(\tau^{-R})
	\]
	as $\tau \to \infty$ where $k, R > 0$ can be chosen arbitrarily large.
	
	There are three differences in the above construction compared with  \cite[Proposition 2.2]{ferreira2013calderon}. The first difference is that $\lambda$ is complex, but the construction of $v_{\tau+i\lambda}$ goes through without changes for complex $\lambda$. The second difference is the factor $\tau^{-\frac{n-2}{8}}$ in front of $v_{L(\tau+i\lambda) }$. The argument in \cite[Proposition 3.1]{ferreira2013calderon} implies that 
	\[
	\norm{v_{L(\tau+i\lambda) }}_{L^2(M_0)} = O(1), \qquad \norm{v_{L(\tau+i\lambda) }}_{L^4(M_0)} = O((L\tau)^{\frac{n-2}{8}}).
	\]
	Hence the main term in the solutions $u_1$ and $u_2$ is normalized so that 
	\[
	\norm{\widetilde{v}_{L(\tau+i\lambda) }}_{L^4(M_0)} = O(1).
	\]
	This normalization will be appropriate for dealing with products of four solutions. The third difference is the decay for the error terms $r_j$, which is better than the decay stated in \cite[Proposition 2.2]{ferreira2013calderon} and can be justified as follows. By the argument in \cite[Proposition 2.2]{ferreira2013calderon}, writing $s=\tau+i\lambda$, the function $r_1$ is obtained by solving the equation 
	\[
	e^{s x_1} \Delta_g (e^{-s x_1} r_1) = \tau^{-\frac{n-2}{8}} (-\Delta_{g_0} - s^2) v_s \text{ in $M$}.
	\]
	By using a Carleman estimate shifted to a negative semiclassical Sobolev space $H^{-k}_{\mathrm{scl}}$, with $h = \tau^{-1}$ as the semiclassical parameter, and using the corresponding solvability result in $H^k_{\mathrm{scl}}$ (see e.g.\ \cite[Lemma 4.3 and Proposition 4.4]{ferreira2009limiting}), there is a solution $r_1 \in H^k(M)$ with 
	\begin{equation} \label{r1_norm_estimate}
		\norm{r_1}_{H^k_{\mathrm{scl}}(M)} \leq \frac{C}{\tau} \norm{(-\Delta_{g_0} - s^2) v_s}_{H^k_{\mathrm{scl}}(M)}.
	\end{equation}
	By \cite[Proposition 3.1]{ferreira2013calderon}, one has $\norm{(-\Delta_{g_0} - s^2) v_s}_{L^2(M_0)} = O(\tau^{-K})$ for some arbitrarily large $K$. However, by looking at the formula for $f = (-\Delta_{g_0} - s^2) v_s$ given after \cite[equation (3.4)]{ferreira2013calderon} we see that taking $k$ derivatives of $f$ has the effect of bringing at most $k$ powers of $s$ to the front of the expression, or reducing the degree of vanishing of $h_j$ on $\Gamma$ by at most $k$. In effect this means that $\norm{(-\Delta_{g_0} - s^2) v_s}_{H^k_{\mathrm{scl}}(M_0)} = O(\tau^{-K})$, leading to 
	\begin{equation} \label{r1_right_hand_side_estimate}
		\norm{(-\Delta_{g_0} - s^2) v_s}_{H^k_{\mathrm{scl}}(M)} = O(\tau^{-K})
	\end{equation}
	where $K$ can be chosen arbitrarily large. The required decay 
	\[
	\norm{r_1}_{H^k(M)} \leq \tau^k \norm{r_1}_{H^k_{\mathrm{scl}}(M)} = O(\tau^{-R})
	\]
	follows by combining \eqref{r1_norm_estimate} and \eqref{r1_right_hand_side_estimate} after choosing $K$ large enough, and the same estimate for $r_2$ follows analogously. It follows from this discussion, after taking $k > n/2$ and using the Sobolev embedding $H^k(M) \subset L^{\infty}(M)$, that 
	\begin{equation} \label{u1_u2_estimate}
	u_1 u_2 = e^{-2i \re(L\lambda) x_1} \abs{\widetilde{v}_{L(\tau+i\lambda) }(x')}^2 + O_{L^2(M)}((L\tau)^{-R}).
	\end{equation}
	
	We now repeat the previous construction for the geodesic $\eta$ and choose solutions 
	\begin{align*}
		u_3 = e^{-(\tau+i\mu) x_1} ( \widetilde{w}_{\tau+i\mu}(x') + r_3), \qquad 
		u_4 = \overline{e^{(\tau+i\mu) x_1} ( \widetilde{w}_{\tau+i\mu}(x') + r_4)},
	\end{align*}
	where $\mu \in \mC$ is fixed, $\norm{r_j}_{L^2(M)} = O(\tau^{-R})$ as $\tau \to \infty$, and 
	\[
	\widetilde{w}_{\tau+i\mu} := \tau^{-\frac{n-2}{8}} w_{\tau+i\mu}(x')
	\]
	where $w_{\tau+i\mu}$ is the Gaussian beam quasimode concentrating near $\eta$ constructed in \cite[Proposition 3.1]{ferreira2013calderon} so that  
	\[
	\norm{\widetilde{w}_{\tau+i\mu}}_{L^4(M_0)} = O(1) \text{ as }\tau \to \infty.
	\]
	Similarly as for $u_1 u_2$, one has 
	\begin{equation} \label{u3_u4_estimate}
	u_3 u_4 = e^{-2i \re(\mu) x_1} \abs{\widetilde{w}_{\tau+i\mu}(x')}^2 + O_{L^2(M)}(\tau^{-R}).
	\end{equation}
	
	\vspace{10pt}

{\it Step 3. The integral of $f$ against $u_1 u_2 u_3 u_4$.} \\

\noindent 
By the assumption that $f$ integrates to zero against products of four harmonic functions, we have 
	\[
	\int_M f u_1 u_2 u_3 u_4 \,dV = 0.
	\]
	Using \eqref{u1_u2_estimate} and \eqref{u3_u4_estimate}, we have  
	\begin{align*}
	&\int_M f u_1 u_2 u_3 u_4 \,dV \\
	 = &\int_M f(x_1,x') e^{-2i \re(L\lambda+\mu)x_1} \abs{\widetilde{v}_{L(\tau+i\lambda) }(x')}^2 \abs{\widetilde{w}_{\tau+i\mu}(x')}^2 \,dV + O(\tau^{-R}).
	\end{align*}
	If we extend $f$ by zero to $\R \times M_0$ and denote the partial Fourier transform of $f$ in the $x_1$ variable by $\hat{f}(\lambda, x')$, the previous identity becomes 
	\[
	\int_M f u_1 u_2 u_3 u_4 \,dV = \int_{M_0} \hat{f}(2 \re(L\lambda+\mu),\,\cdot\,)  \abs{\widetilde{v}_{L(\tau+i\lambda) }}^2 \abs{\widetilde{w}_{\tau+i\mu}}^2 \,dV_{g_0} + O(\tau^{-R}).
	\]
	Note that $\widetilde{v}_{L(\tau+i\lambda) }$ and $\widetilde{w}_{\tau+i\mu}$ can be chosen to be supported in arbitrarily small but fixed neighborhoods of $\gamma$ and $\eta$, respectively. Thus if $p_1, \ldots, p_N$ are the distinct points in $M_0$ where $\gamma$ intersects $\eta$, then the last integral over $M_0$ is actually over $U_1 \cup \ldots \cup U_N$ where $U_r$ is a small neighborhood of $p_r$ in $M_0$. 
	
	We will use the abbreviation 
	\[
	F(x') = F_{\re(L\lambda+\mu)}(x') := \hat{f}(2 \re(L\lambda+\mu),x').
	\]
	Note for later purposes that $F$ is independent of $\im(\lambda)$ and $\im(\mu)$ and that 
	\[
	\norm{F}_{C^1(M_0)} \lesssim \norm{f}_{C^1(M)}.
	\]
	Combining the above facts, we have 
	\begin{equation} \label{f_integral_limit_intermediate}
	\sum_{r=1}^N \tau^{\frac{1}{2}} \int_{U_r} F \abs{\widetilde{v}_{L(\tau+i\lambda) }}^2 \abs{\widetilde{w}_{\tau+i\mu}}^2 \,dV_{g_0} = o(1)
	\end{equation}
	as $\tau \to \infty$. It will be shown below that with the normalizing factor $\tau^{\frac{1}{2}}$, the left hand side has a nontrivial limit as $\tau \to \infty$.
	
	\vspace{10pt}

{\it Step 4. Analysis of the integrals in \eqref{f_integral_limit_intermediate}.} \\

\noindent 
	Fix now $p$ to be one of the intersection points $p_r$ and let $U = U_r$. We consider the integral 
	\begin{equation} \label{integral_intermediate}
	\int_{U} F \abs{\widetilde{v}_{L(\tau+i\lambda) }}^2 \abs{\widetilde{w}_{\tau+i\mu}}^2 \,dV_{g_0}.
	\end{equation}
	By the construction in \cite[Proposition 3.1]{ferreira2013calderon}, in $U$ the quasimode $\widetilde{v}_{\tau+i\lambda}$ is a finite sum 
	\[
	\widetilde{v}_{L(\tau+i\lambda) }|_U = v^{(1)} + \ldots + v^{(P)}
	\]
	where $t_1 < \ldots < t_P$ are the times in $[0,T]$ when $\gamma(t_j) = p$, each $v^{(j)}$ has the form 
	\[
	v^{(j)} = \tau^{\frac{n-2}{8}} e^{iL(\tau+i\lambda) \Phi_j} a_j
	\]
	where each $\Phi = \Phi_j$ is a smooth complex function in $U$ satisfying, for $t$ close to $t_j$,  
	\begin{align*}
		\Phi(\gamma(t)) = t, \quad \nabla \Phi(\gamma(t)) = \dot{\gamma}(t), \quad \mathrm{Im}(\nabla^2 \Phi(\gamma(t))) \geq 0, \quad \mathrm{Im}(\nabla^2 \Phi)|_{\dot{\gamma}(t)^{\perp}} > 0,
	\end{align*}
	and each $a_j$ is a smooth function in $U$ of the form 
	\[
	a_j(t,y) = (a_{0,j}(t,y) + O(\tau^{-1})) \chi(y/\delta')
	\]
	where $(t,y)$ are Fermi coordinates for $\gamma$ for $t$ close to $t_j$, $a_{0,j}(t,0)$ is a nonvanishing function independent of $\tau$ and $\lambda$, $\chi$ is a smooth cutoff function supported in the unit ball, and $\delta' > 0$ is a fixed number that can be taken very small. Note that as opposed to \cite{ferreira2013calderon}, there is no power of $\tau$ in the definition of $a_j$. By the argument in \cite[formula (3.5) and before]{ferreira2013calderon} one also has, as $\tau \to \infty$, 
	\begin{equation} \label{vj_bounds}
	\norm{v^{(j)}}_{L^4(U)} = O(1), \qquad \norm{v^{(j)}}_{L^4(U \cap \p M_0)} = O(1).
	\end{equation}
	
	In a similar way, $\widetilde{w}_{\tau+i\mu}$ in $U$ is a finite sum 
	\[
	\widetilde{w}_{\tau+i\mu}|_U = w^{(1)} + \ldots + w^{(Q)}
	\]
	where $s_1 < \ldots < s_Q$ are the times in $[0,S]$ when $\eta(s_k) = p$, each $w^{(k)}$ has the form 
	\[
	w^{(k)} = \tau^{\frac{n-2}{8}} e^{i(\tau+i\mu) \Psi_k} b_k
	\]
	where each $\Psi = \Psi_k$ satisfies, for $s$ close to $t_k$,  
	\begin{align*}
		\Psi(\eta(s)) = s, \quad \nabla \Psi(\eta(s)) = \dot{\eta}(s), \quad \mathrm{Im}(\nabla^2 \Psi(\eta(s)) \geq 0, \quad \mathrm{Im}(\nabla^2 \Psi)|_{\dot{\eta}(s)^{\perp}} > 0,
	\end{align*}
	and $b_j$ has a similar form as $a_j$ but is supported near $\eta$.
	
	Inserting the formulas for $\widetilde{v}_{L(\tau+i\lambda) }$ and $\widetilde{w}_{\tau+i\mu}$ in \eqref{integral_intermediate} yields that 
	\begin{equation} \label{integral_ijklm_sum}
	\tau^{\frac{1}{2}} \int_{U} F \abs{\widetilde{v}_{L(\tau+i\lambda) }}^2 \abs{\widetilde{w}_{\tau+i\mu}}^2 \,dV_{g_0} = \sum_{j,k,l,m} I_{jklm}
	\end{equation}
	where 
	\begin{align*}
	I_{jklm} &= \tau^{\frac{1}{2}} \int_{U} F v^{(j)} \ol{v^{(k)}} w^{(l)} \ol{w^{(m)}} \,dV_{g_0} \\
	 &= \tau^{\frac{n-1}{2}} \int_{U} e^{i\tau \Xi_{jklm}} A_{jklm} F \,dV_{g_0}
	\end{align*}
	where 
	\begin{align*}
	\Xi_{jklm} &:=L \Phi_j-L\ol{\Phi}_k+\Psi_l-\ol{\Psi}_m, \\
	A_{jklm} &:= e^{-L\lambda \Phi_j-L\ol{\lambda \Phi_k}} e^{-\mu \Psi_l - \ol{\mu \Psi_m}} a_j \ol{a}_k b_l \ol{b}_m.
	\end{align*}
	We will next analyze the integrals $I_{jklm}$ and show that the only nontrivial contributions as $\tau \to \infty$ come from the terms where $d\Xi_{jklm}(p) = 0$. After this, we will fix the parameters so that $d\Xi_{jklm}(p) = 0$ will happen only when $j=k$ and $l=m$.
	
	\vspace{10pt}
	
{\it Step 5. Evaluation of $I_{jklm}$ when $d\Xi_{jklm}(p) = 0$.} \\

\noindent 
	Let $j,k,l,m$ be such that $\Xi = \Xi_{jklm}$ satisfies $d\Xi(p) = 0$, and write 
	\[
	B := F A_{jklm}.
	\]
	Writing $z$ for the geodesic normal coordinates in $(M_0,g_0)$ with origin at $p$, the phase function $\Xi$ has the Taylor expansion 
	\[
	\Xi(z) = \Xi(0) + \frac{1}{2} H z \cdot z + O(\abs{z}^3).
	\]
	Here $\Xi(0) = L(t_j - t_k) + s_l - s_m$ and $H = H_{jklm} = (\p_{z_a z_b} \Xi)_{a,b}$ is the Hessian of $\Xi$ in the $z$ coordinates. Note that the imaginary parts of Hessians of $\Phi_j, \Phi_k, \Psi_l, \Psi_m$ at $p$ are all positive semidefinite. Moreover, they are positive definite in the codimension one subspaces $\dot{\gamma}(t_j)^{\perp}, \dot{\gamma}(t_k)^{\perp}, \dot{\eta}(s_l)^{\perp}, \dot{\eta}(s_m)^{\perp}$, respectively. Thus it follows that $\im(H) = \im(\nabla^2(L (\Phi_j+\Phi_k)+\Psi_l+\Psi_m))|_p$ is positive semidefinite, and moreover it is positive definite since the above codimension one subspaces span the whole tangent space at $p$. The last fact holds since $\dot{\gamma}(t_j) \neq \pm \dot{\eta}(s_l)$, which follows because the geodesics $\gamma$ and $\eta$ are distinct and one is not the reverse of the other. Finally, since $\im(H)$ is positive definite, by choosing $U$ small one has $\abs{z} \leq \delta$ in $U$ where $\delta$ is very small, which implies that 
	\[
	\abs{e^{i\tau(\frac{1}{2} H z \cdot z + O(\abs{z}^3))}} \leq e^{-c \tau \abs{z}^2}
	\]
	for some $c > 0$. This shows that one may indeed use dominated convergence in the argument below.
	
	One has 
	\begin{align*}
	I_{jklm} = &\tau^{\frac{n-1}{2}} \int_U e^{i\tau \Xi} B \,dV_{g_0} \\
	=& \tau^{\frac{n-1}{2}} e^{i\tau \Xi(0)} \int_{\mR^{n-1}} e^{i\tau (\frac{1}{2} H z \cdot z + O(\abs{z}^3))} B(z) \abs{g_0(z)}^{1/2} \,dz.
	\end{align*}
	We will change variables $z \to \tau^{-1/2} z$, which brings a Jacobian $\tau^{-\frac{n-1}{2}}$ that cancels the power of $\tau$ in front.
	Note that one has $\abs{g_0(z/\tau^{1/2})} \to 1$ and 
	\[
	B(z/\tau^{1/2}) \to F(p) e^{-L \lambda t_j - L \bar{\lambda} t_k} e^{-\mu s_l - \bar{\mu} s_m} a_j(p) \ol{a_k(p)} b_l(p) \ol{b_m(p)}
	\]
	as $\tau \to \infty$. Combining these facts and using dominated convergence gives that 
	\[
	I_{jklm} = e^{i\tau(L(t_j - t_k) + s_l - s_m)}  c_{jklm} F(p) e^{-L \lambda t_j - L \bar{\lambda} t_k} e^{-\mu s_l - \bar{\mu} s_m} + o(1)
	\]
	where 
	\[
	c_{jklm} = a_j(p) \ol{a_k(p)} b_l(p) \ol{b_m(p)} \int_{\mR^{n-1}} e^{\frac{i}{2} H_{jklm} z \cdot z} \,dz.
	\]
	The last integral is finite since $\im(H_{jklm})$ is positive definite. For later purposes we observe that $H_{jjll}$ is purely imaginary, hence 
	\[
	c_{jjll} = \abs{a_j(p)}^2 \abs{b_l(p)}^2 \int_{\mR^{n-1}} e^{-\frac{1}{2}\im(H_{jjll}) z \cdot z} \,dz
	\]
	where the last integral is positive. In particular, we have 
	\[
	I_{jjll} = c_{jjll} F(p) e^{-2L\re(\lambda) t_j} e^{-2\re(\mu) s_l} + o(1)
	\]
	where $c_{jjll} > 0$.
	
	\vspace{10pt}

{\it Step 6. Evaluation of $I_{jklm}$ when $d\Xi_{jklm}(p) \neq 0$.} \\

\noindent 
	Write $\varphi = \re(\Xi_{jklm})$. Since $d\Phi_j(p), d\Psi_l(p)$ etc are real, we have $d\varphi(p) \neq 0$, and $I_{jklm}$ may be written as 
	\[
	I_{jklm} = \tau^{\frac{1}{2}} \int_{U} e^{i\tau \varphi} F \breve{v}^{(j)} \ol{\breve{v}^{(k)}} \breve{w}^{(l)} \ol{\breve{w}^{(m)}} \,dV_{g_0}
	\]
	where 
	\[
	\breve{v}^{(j)} = \tau^{\frac{n-2}{8}} e^{-L \tau \im(\Phi_j) - L\lambda \Phi_j} a_j, \qquad \breve{w}^{(l)} = \tau^{\frac{n-2}{8}} e^{-\tau \im(\Psi_l) -\mu \Psi_l} b_l.
	\]
	We wish to use a non-stationary phase argument as in \cite[end of proof Proposition 3.1]{ferreira2013calderon}. Write 
	\[
	e^{i\tau \varphi} = \frac{1}{i\tau} P(e^{i\tau\varphi}), \qquad Pw = \langle \abs{d\varphi}^{-2} d\varphi, dw \rangle,
	\]
	where we assume that $U$ has been chosen so small that $d\varphi$ is nonvanishing in $U$. Since $F$ is $C^1$, we may integrate by parts to obtain 
	\begin{align}
	I_{jklm} = &\notag  \frac{1}{i\tau^{\frac{1}{2}}} \int_{U} e^{i\tau \varphi} P^t \left[ F \breve{v}^{(j)} \ol{\breve{v}^{(k)}} \breve{w}^{(l)} \ol{\breve{w}^{(m)}} \right] \,dV_{g_0} \\
	 & \quad+ \frac{1}{i\tau^{\frac{1}{2}}} \int_{U \cap \p M} \frac{\p_{\nu} \varphi}{\abs{d\varphi}^2} e^{i\tau \varphi} F \breve{v}^{(j)} \ol{\breve{v}^{(k)}} \breve{w}^{(l)} \ol{\breve{w}^{(m)}} \,dS \label{ijklm_nonstationary}
	\end{align}
	where the boundary term only appears if $p \in \p M$.
	
	The boundary term in \eqref{ijklm_nonstationary} goes to zero as $\tau \to \infty$ since $\norm{\breve{v}^{(j)}}_{L^4(U \cap \p M)} \lesssim 1$ by \eqref{vj_bounds} etc. In the integral over $U$, if the derivative hits $F$ one can estimate 
	\begin{align*}
	&\abs{\frac{1}{i\tau^{\frac{1}{2}}} \int_{U} e^{i\tau \varphi} (P^t F) \breve{v}^{(j)} \ol{\breve{v}^{(k)}} \breve{w}^{(l)} \ol{\breve{w}^{(m)}} \,dV_{g_0}} \\
	& \qquad \lesssim \tau^{-1/2} \norm{F}_{C^1(M)} \norm{\breve{v}^{(j)} \ol{\breve{v}^{(k)}} \breve{w}^{(l)} \ol{\breve{w}^{(m)}}}_{L^1(M_0)}
	\end{align*}
	which goes to zero as $\tau \to \infty$ since $\norm{\breve{v}^{(j)}}_{L^4(M_0)} \lesssim 1$ etc. The worst behaviour in $\tau$ in the integral over $U$ occurs when a derivative hits $e^{-L \tau \im(\Phi_j)}$ or $e^{-\tau \im(\Psi_l)}$ etc. This brings a factor like $\tau d(\im(\Phi_j))$ into the integrand, and since $d(\im(\Phi_j))$ vanishes on $\gamma$ one can choose new coordinates $z = (z', z_{n-1})$ near $0$ so that $\abs{d(\im(\Phi_j))} \lesssim \abs{z'}$. Thus the integral that one needs to estimate looks like 
	\[
	\frac{1}{i\tau^{\frac{1}{2}}} \int_{U} e^{i\tau \varphi} F \left[ \tau d(\im(\Phi_j)) \right] \breve{v}^{(j)} \ol{\breve{v}^{(k)}} \breve{w}^{(l)} \ol{\breve{w}^{(m)}} \,dV_{g_0}.
	\]
	Evaluating this integral as in Step 5, and using that the change of variables $z \to \tau^{-1/2} z$ together with $\abs{d(\im(\Phi_j))} \lesssim \abs{z'}$ brings an additional factor $\tau^{-1/2}$, shows that this kind of integral is $O(\tau^{-1/2})$. 
	This concludes the proof that 
	\begin{align*}\label{limit tends to zero of I jklm}
	d\Xi_{jklm}(p) \neq 0 \implies \lim_{\tau \to \infty} I_{jklm} = 0.
	\end{align*}
	
	\vspace{10pt}

{\it Step 7. Evaluation of \eqref{f_integral_limit_intermediate}.} \\

\noindent 
	Recall from Step 3 that $p_1, \ldots, p_N$ were the distinct intersection points of $\gamma$ and $\eta$ and that $U_r$ were small neighborhoods of $p_r$. As in Step 4, for each $r$ with $1 \leq r \leq N$ let $t_1^{(r)} < \ldots < t_{P_r}^{(r)}$ be the times in $[0,T]$ when $\gamma(t_j^{(r)}) = p_r$, and let $s_1^{(r)} < \ldots < s_{Q_r}^{(r)}$ be the times in $[0,S]$ when $\eta(s_j^{(r)}) = p_r$.
	
	Going back to \eqref{f_integral_limit_intermediate} and using \eqref{integral_ijklm_sum}, we have 
	\[
	\sum_{r=1}^N \sum_{j,k=1}^{P_r} \sum_{l,m=1}^{Q_r} I_{jklm}^{(r)} = o(1)
	\]
	as $\tau \to \infty$, where 
	\[
	I_{jklm}^{(r)} = \tau^{\frac{n-1}{2}} \int_{U_r} e^{i\tau \Xi_{jklm}^{(r)}} A_{jklm}^{(r)} F \,dV_{g_0}
	\]
	where $\Xi_{jklm}^{(r)}$ and $A_{jklm}^{(r)}$ are defined in $U_r$ as in Step 4.
	
	The integrals $I_{jklm}^{(r)}$ were evaluated in Steps 5 and 6. If we define 
	\[
	c_{jklm}^{(r)} := \left\{ \begin{array}{cl} a_j^{(r)}(p_r) \ol{a_k^{(r)}(p_r)} b_l^{(r)}(p_r) \ol{b_m^{(r)}(p_r)} \int_{\mR^{n-1}} e^{\frac{i}{2} H_{jklm}^{(r)} z \cdot z} \,dz, & d\Xi_{jklm}^{(r)}(p_r) = 0, \\ 0, & \text{otherwise}, \end{array} \right.
	\]
	then we get from Steps 5 and 6 that 
	\begin{align}\label{total sum....}
	\sum_{r=1}^N \sum_{j,k=1}^{P_r} \sum_{l,m=1}^{Q_r} e^{i\tau \left[ L( t_j^{(r)} - t_k^{(r)} )+ s_l^{(r)} - s_m^{(r)} \right]}  c_{jklm}^{(r)} F(p_r) e^{-L(\lambda t_j^{(r)} + \bar{\lambda} t_k^{(r)})} e^{-\mu s_l^{(r)} - \bar{\mu} s_m^{(r)}} = o(1)
	\end{align}
	as $\tau \to \infty$. In the above formula $F(p_r) = F_{\re(\lambda+\mu)}(p_r)$.
	
	\vspace{10pt}

	{\it Step 7. Choosing $L$ so that $d\Xi_{jklm}^{(r)}(p_r) = 0$ only when $j=k$ and $l=m$.} \\
	
	
	Next, we choose $L\in \N$ large enough, but fixed, so that 
	\[
	d\Xi_{jklm}(p_r) \neq 0
	\]
	for all $1\leq r\leq N$ unless $j=k$ and $l=m$. We have
	\begin{align*}
	\nabla \Xi_{jklm}(p_r)&=L\nabla \Phi_j-L\nabla\ol{\Phi}_k+\nabla\Psi_l-\nabla\ol{\Psi}_m \\ &=L(\dot{\gamma}(t_j^{(r)})-\dot{\gamma}(t_k^{(r)}))+\dot{\eta}(s_l^{(r)})-\dot{\eta}(s_m^{(r)}).
	\end{align*}
	Since the geodesic $\gamma$ is transversal and thus it is not closed geodesic, we have
	\[
	\dot{\gamma}(t_j^{(r)})-\dot{\gamma}(t_k^{(r)})\neq 0
	\]
	for all $j\neq k$ for all $r$ with $1\leq r\leq N$.  Let 
	\[
	\alpha=\min\left\{ \abs{\dot{\gamma}(t_j^{(r)})-\dot{\gamma}(t_k^{(r)})}:\ 1\leq r\leq N,\  1\leq j,k\leq P_r \right\} >0
	\]
	and
	\[
	\beta=\max\left\{\abs{\dot{\eta}(s_l^{(r)})-\dot{\eta}(s_m^{(r)})}:\ 1\leq r\leq N, \  1\leq l,m \leq Q_r  \right\}.
	\]
	We choose $$L \geq \frac{\beta+1}{\alpha},$$ then we have 
	\begin{align*}
	\abs{L(\dot{\gamma}(t_j^{(r)})-\dot{\gamma}(t_k^{(r)}))+\dot{\eta}(s_l^{(r)})-\dot{\eta}(s_m^{(r)})} &\geq  \abs{L(\dot{\gamma}(t_j)-\dot{\gamma}(t_k))}-\abs{\dot{\eta}(s_l)-\dot{\eta}(s_m)}\\
	 &\geq L \alpha  -\beta \geq 1>0.
	\end{align*}
	Thus we have found that for any $L\geq \frac{\beta+1}{\alpha}$, one has 
	\[
	d\Xi_{jklm}(p_r) \neq 0
	\]
	if $j\neq k$. 
	
	For $1\leq r \leq N$, assume then that $j=k$ and $l\neq m$. Then we have
	\[
	\nabla \Xi_{jklm}(p_r)=L(\dot{\gamma}(t_j^{(r)})-\dot{\gamma}(t_j^{(r)}))+\dot{\eta}(s_l^{(r)})-\dot{\eta}(s_m^{(r)})\neq 0,
	\]
	since $\eta$ is transversal and $l\neq m$.
	
	In conclusion, the only case when $d\Xi _{jklm}^{(r)}(p_r)=0$ is $j=k$ and $l=m$.

	\vspace{10pt}

{\it Step 8. Conclusion of the proof.} \\

\noindent 
	Going back to \eqref{total sum....} and using the result in Step 7, and taking $\tau \to \infty$, we have 
	\begin{align}\label{total sum =0}
	\sum_{r=1}^N \sum_{j=1}^{P_r} \sum_{l=1}^{Q_r}  c_{jjll}^{(r)} F_{\re(L\lambda+\mu)}(p_r) e^{-L t_j^{(r)} (\lambda+ \bar{\lambda})} e^{s_l^{(r)} (-\mu - \bar{\mu}) } =  0,
	\end{align}
	where we have used that $e^{i\tau \left[ L( t_j^{(r)} - t_k^{(r)} )+ s_l^{(r)} - s_m^{(r)} \right]}=1$, when $j=k$ and $l=m$.
	
Let $\lambda \in \R$ and choose $\mu$ so that 
\[
2L\lambda +2\mu =2\lambda,
\]
which is equivalent to  
\[
\mu=(1-L)\lambda \in \R.
\]
Then \eqref{total sum =0} reads 
\begin{equation} \label{sum_l_lambda}
\sum_{r=1}^N \sum_{j=1}^{P_r} \sum_{l=1}^{Q_r}  c_{jjll}^{(r)} F_{\lambda}(p_r) e^{-2\lambda \left[L \left(t_j^{(r)}-s_l^{(r)}\right) +s_l^{(r)}\right] } = 0,
\end{equation}

We define two sets of real numbers as follows:
\[
\mathcal{Q}_1=\cup_{r,r'=1}^N\cup_{j,k=1}^{P_r}\left\{t_j^{(r)}-t_k^{(r')}\right\},  \quad \quad \mathcal{Q}_2=\cup_{r,r'=1}^N\cup_{l,m=1}^{Q_r}\left\{s_l^{(r)}-s_m^{(r')}\right\},
\]
and 
\[
\widetilde \alpha=\min_{d_1\in \mathcal{Q}_1, d_2\in\mathcal{Q}_2, d_1\neq d_2}\abs{d_1-d_2}, \quad \quad \widetilde \beta =\max_{d\in\mathcal{Q}_2}\abs{d}.
\]
Finally, we redefine the number $L$ in Step 7 as 
\[
L=\max \left\{\frac{\widetilde \beta}{\widetilde \alpha}+1, \frac{\beta+1}{\alpha} \right\},
\]
where $\alpha$, $\beta$ are the numbers given in Step 7.

Let $(r_1,j_1,l_1)\neq (r_2,j_2,l_2)$, then we want to show that 
\begin{equation}\label{not_same}
L(t_{j_1}^{(r_1)}-s_{l_1}^{(r_1)})+s_{l_1}^{(r_1)} \neq L(t_{j_2}^{(r_2)}-s_{l_2}^{(r_2)})+s_{l_2}^{(r_2)}.
\end{equation}
Let us set 
\[
d_1=t_{j_1}^{(r_1)}-t_{j_2}^{(r_2)} \text{ and } d_2=s_{l_1}^{(r_1)}-s_{l_2}^{(r_2)}.
\]
We have following cases:

\begin{itemize}
	\item[(a)] For the case $d_1=d_2$. Suppose that 
	\begin{equation}\label{the_same}
	\L(t_{j_1}^{(r_1)}-s_{l_1}^{(r_1)})+s_{l_1}^{(r_1)} = L(t_{j_2}^{(r_2)}-s_{l_2}^{(r_2)})+s_{l_2}^{(r_2)},
	\end{equation}
	then we have $L(d_1-d_2)+d_2=0$. 
	It follows that $s_{l_1}^{(r_1)}=s_{l_2}^{(r_2)}$. Thus $l_1=l_2$ and $r_1=r_2$. Since $d_1=d_2$, we also have $t_{j_1}^{(r_1)}=t_{j_2}^{(r_2)}$. Thus, $j_1=j_2$ holds, which leads to a contradiction to $(r_1, j_1,l_1)\neq (r_2,j_2,l_2)$. Thus we must have~\eqref{not_same}.
	
	\item[(b)] For the case $d_1\neq d_2$. If ~\eqref{the_same} holds, we have
	\[
	L=\frac{d_2}{d_1-d_2}.
	\]
	However, this cannot be true since
	\[
	L=\frac{\widetilde \beta}{\widetilde \alpha}+1>\frac{d_2}{d_1-d_2}.
	\]
	Thus again we have~\eqref{not_same}.
\end{itemize}
We now go back to \eqref{sum_l_lambda} and use Lemma~\ref{lemma_linear_combination_analytic} below together with the condition \eqref{not_same} and the fact that $F_{\lambda}(p_r) = \hat{f}(2\lambda,p_r)$. This yields that 
\[
c_{jjll}^{(r)} \hat{f}(2\lambda,p_r) = 0
\]
for all $(r, j, l)$. In Step 5 we proved that $c_{jjll}^{(r)} > 0$ for all $(r,j,l)$, showing that $\hat{f}(2\lambda,p_r) = 0$ for all $\lambda \in \mR$ and all $r$. Since the point $y_0$ in Step 1 was one of the $p_r$, it follows that $f(x_1,y_0) = 0$ for all $x_1 \in \mR$. By Step 1 this is true for almost every $y_0$ in $(M_0,g_0)$, and by the continuity of $f$ one gets that $f \equiv 0$ as required.
\end{proof}

\begin{lem} \label{lemma_linear_combination_analytic}
	Let $f_1, \ldots, f_N$ be compactly supported distributions in $\R$ such that for some distinct real numbers $a_1, \ldots, a_N$ one has 
	\[
	\sum_{j=1}^N \hat{f}_j(\lambda) e^{a_j \lambda} = 0, \qquad \lambda \in \R.
	\]
	Then $f_1 = \ldots = f_N = 0$.
\end{lem}

\begin{proof}
	Suppose without loss of generality that $a_1 > a_2 > \ldots > a_N$. Then 
	\[
	\hat{f}_1(\lambda) = -\sum_{j=2}^N e^{-(a_1-a_j) \lambda} \hat{f}_j(\lambda).
	\]
	By the Paley-Wiener-Schwartz theorem there are $C, M > 0$ so that 
	\[
	\abs{\hat{f}_j(\lambda)} \leq C (1+\abs{\lambda})^M, \qquad \lambda \in \R.
	\]
	Write $\delta = a_1-a_2 > 0$. Since $a_1-a_j \geq \delta$ for $j \geq 2$, it follows that 
	\begin{align*}
		\abs{\hat{f}_1(\lambda)} \leq \begin{cases}
			C(1+|\lambda|)^M, \ &\lambda\leq 0,\\
			C(1+\lambda)^M e^{-\delta \lambda}, \ &\lambda\geq 0. 
		\end{cases}
	\end{align*}

	However, no nontrivial compactly supported distribution $f_1$ can have the above decay for its Fourier transform. To see this, note that 
	\[
	e^{\eps \lambda} \hat{f}_1(\lambda) \in \mathscr{S}'(\R), \qquad 0 \leq \eps < \delta.
	\]
	Thus using \cite[Theorem 7.4.2]{hormander1983analysis} there exists an analytic function $U$ in $\{0 < \im(t) < \delta \}$ so that the Fourier transform of $e^{\eps \lambda} \hat{f}_1(\lambda)$ is $U(\,\cdot\,+i\eps)$. By \cite[Remark after Theorem 7.4.3]{hormander1983analysis} the limit of $U(\,\cdot\,+i\eps)$ in $\mathscr{S}'(\R)$ as $\eps \to 0$ is the Fourier transform of $\hat{f}_1(\lambda)$, i.e.\ $2\pi f_1(-\,\cdot\,)$. Fix some interval $I \subset \R$ that is outside the support of $f_1(-\,\cdot\,)$, and consider the rectangle $Z = I \times (0,\delta)$. For $\eps$ close to $0$, one has 
	\[
	\abs{U(t+i\eps)} \leq \norm{e^{\eps \lambda} \hat{f_1}(\lambda)}_{L^1} \lesssim \norm{(1+\abs{\lambda})^M e^{\eps \lambda}}_{L^1(\R_-)} + 1 \lesssim \eps^{-M}.
	\]
	Since the limit of $U(\,\cdot\,+i\eps)$ in $\mathscr{D}'(I)$ is $2\pi f_1(-\,\cdot\,)|_I = 0$, by \cite[Theorem 3.1.15]{hormander1983analysis} one has $U = 0$ in $Z$. Now $U$ is analytic, so $U = 0$ in $\{0 < \im(t) < \delta \}$ and $f_1 = 0$. Repeating this argument gives that $f_2 = \ldots = f_N = 0$.
\end{proof}

\bibliographystyle{alpha}
\bibliography{ref}

\newcommand{\etalchar}[1]{$^{#1}$}
\begin{thebibliography}{dHUW18}

\bibitem[BHKS18]{branderetal_monotonicity_plaplace}
Tommi Brander, Bastian Harrach, Manas Kar, and Mikko Salo.
\newblock Monotonicity and enclosure methods for the {$p$}-{L}aplace equation.
\newblock {\em SIAM J. Appl. Math.}, 78(2):742--758, 2018.

\bibitem[Cal80]{calderon2006inverse}
Alberto~P. Calder\'{o}n.
\newblock On an inverse boundary value problem.
\newblock In {\em Seminar on {N}umerical {A}nalysis and its {A}pplications to
  {C}ontinuum {P}hysics ({R}io de {J}aneiro, 1980)}, pages 65--73. Soc. Brasil.
  Mat., Rio de Janeiro, 1980.
\newblock Reprinted in Computational \& Applied Mathematics 25 (2016), no.
  2--3, 133--138.

\bibitem[CLOP19]{chen2019detection}
Xi~Chen, Matti Lassas, Lauri Oksanen, and Gabriel~P Paternain.
\newblock Detection of {H}ermitian connections in wave equations with cubic
  non-linearity.
\newblock {\em arXiv preprint arXiv:1902.05711}, 2019.

\bibitem[CNV19]{carsteanakamuravashisth}
C{\u{a}}t{\u{a}}lin~I C{\^a}rstea, Gen Nakamura, and Manmohan Vashisth.
\newblock Reconstruction for the coefficients of a quasilinear elliptic partial
  differential equation.
\newblock {\em arXiv preprint arXiv:1903.07034}, 2019.

\bibitem[dHUW18]{de2018nonlinear}
Maarten de~Hoop, Gunther Uhlmann, and Yiran Wang.
\newblock Nonlinear interaction of waves in elastodynamics and an inverse
  problem.
\newblock {\em Mathematische Annalen}, pages 1--31, 2018.

\bibitem[FKL{\etalchar{+}}17]{ferreira2017linearized}
David Dos~Santos Ferreira, Yaroslav Kurylev, Matti Lassas, Tony Liimatainen,
  and Mikko Salo.
\newblock The linearized {C}alder\'on problem in transversally anisotropic
  geometries.
\newblock {\em arXiv preprint arXiv:1712.04716 (to appear IMRN)}, 2017.

\bibitem[FKLS16]{ferreira2013calderon}
David Dos~Santos Ferreira, Yaroslav Kurylev, Matti Lassas, and Mikko Salo.
\newblock The {C}alder{\'o}n problem in transversally anisotropic geometries.
\newblock {\em J. Eur. Math. Soc. (JEMS)}, 18:2579--2626, 2016.

\bibitem[FKSU09]{ferreira2009limiting}
David Dos~Santos Ferreira, Carlos Kenig, Mikko Salo, and Gunther Uhlmann.
\newblock Limiting {C}arleman weights and anisotropic inverse problems.
\newblock {\em Inventiones mathematicae}, 178(1):119--171, 2009.

\bibitem[GST18]{guillarmousalotzou_complex}
Colin {Guillarmou}, Mikko {Salo}, and Leo {Tzou}.
\newblock {The linearized {C}alder\'on problem on complex manifolds}.
\newblock {\em arXiv e-prints}, page arXiv:1805.00752, May 2018.

\bibitem[GT11]{guillarmou2011calderon}
Colin Guillarmou and Leo Tzou.
\newblock Calder{\'o}n inverse problem with partial data on {R}iemann surfaces.
\newblock {\em Duke Mathematical Journal}, 158(1):83--120, 2011.

\bibitem[GT13]{guillarmoutzou2013_survey}
Colin Guillarmou and Leo Tzou.
\newblock The {C}alder\'{o}n inverse problem in two dimensions.
\newblock In {\em Inverse problems and applications: inside out. {II}},
  volume~60 of {\em Math. Sci. Res. Inst. Publ.}, pages 119--166. Cambridge
  Univ. Press, Cambridge, 2013.

\bibitem[Hor85]{hormander1983analysis}
Lars Hormander.
\newblock {\em The Analysis of Linear Partial Differential Operators. {I-IV}}.
\newblock 1983-1985.

\bibitem[IN95]{VictorN}
Victor Isakov and A~Nachman.
\newblock Global uniqueness for a two-dimensional elliptic inverse problem.
\newblock {\em Trans.of AMS}, 347:3375--3391, 1995.

\bibitem[IS94]{isakov1994global}
Victor Isakov and John Sylvester.
\newblock Global uniqueness for a semilinear elliptic inverse problem.
\newblock {\em Communications on Pure and Applied Mathematics},
  47(10):1403--1410, 1994.

\bibitem[Isa93]{isakov1993uniqueness}
Victor Isakov.
\newblock On uniqueness in inverse problems for semilinear parabolic equations.
\newblock {\em Archive for Rational Mechanics and Analysis}, 124(1):1--12,
  1993.

\bibitem[IUY12]{iuy_general}
Oleg Imanuvilov, Gunther Uhlmann, and Masahiro Yamamoto.
\newblock Partial {C}auchy data for general second order elliptic operators in
  two dimensions.
\newblock {\em Publ. Res. Inst. Math. Sci.}, 48(4):971--1055, 2012.

\bibitem[IY13]{imanuvilovyamamoto_semilinear}
Oleg Imanuvilov and Masahiro Yamamoto.
\newblock Unique determination of potentials and semilinear terms of semilinear
  elliptic equations from partial {C}auchy data.
\newblock {\em J. Inverse Ill-Posed Probl.}, 21(1):85--108, 2013.

\bibitem[KLOU14]{kurylev2014einstein}
Yaroslav Kurylev, Matti Lassas, Lauri Oksanen, and Gunther Uhlmann.
\newblock Inverse problem for einstein-scalar field equations.
\newblock {\em arXiv preprint arXiv:1406.4776}, 2014.

\bibitem[KLU18]{kurylev2018inverse}
Yaroslav Kurylev, Matti Lassas, and Gunther Uhlmann.
\newblock Inverse problems for {L}orentzian manifolds and non-linear hyperbolic
  equations.
\newblock {\em Inventiones mathematicae}, 212(3):781--857, 2018.

\bibitem[KN02]{kang2002identification}
Hyeonbae Kang and Gen Nakamura.
\newblock Identification of nonlinearity in a conductivity equation via the
  {D}irichlet-to-{N}eumann map.
\newblock {\em Inverse Problems}, 18(4):1079, 2002.

\bibitem[KS13]{kenigsalo_partial}
Carlos Kenig and Mikko Salo.
\newblock The {C}alder\'{o}n problem with partial data on manifolds and
  applications.
\newblock {\em Anal. PDE}, 6(8):2003--2048, 2013.

\bibitem[LL19]{lai2019global}
Ru-Yu Lai and Yi-Hsuan Lin.
\newblock Global uniqueness for the fractional semilinear {S}chr{\"o}dinger
  equation.
\newblock {\em Proceedings of the American Mathematical Society},
  147(3):1189--1199, 2019.

\bibitem[LLS18]{lassas2018poisson}
Matti Lassas, Tony Liimatainen, and Mikko Salo.
\newblock The {P}oisson embedding approach to the {C}alder\'on problem.
\newblock {\em arXiv preprint arXiv:1806.04954 (to appear Mathematische
  Annalen)}, 2018.

\bibitem[LU01]{lassas2001determining}
Matti Lassas and Gunther Uhlmann.
\newblock On determining a {R}iemannian manifold from the
  {D}irichlet-to-{N}eumann map.
\newblock In {\em Annales Scientifiques de L’Ecole Normale Superieure},
  volume~34, pages 771--787. No longer published by Elsevier, 2001.

\bibitem[LUW17]{lassas2017determination}
Matti Lassas, Gunther Uhlmann, and Yiran Wang.
\newblock Determination of vacuum space-times from the {E}instein-{M}axwell
  equations.
\newblock {\em arXiv preprint arXiv:1703.10704}, 2017.

\bibitem[LUW18]{lassas2018inverse}
Matti Lassas, Gunther Uhlmann, and Yiran Wang.
\newblock Inverse problems for semilinear wave equations on {L}orentzian
  manifolds.
\newblock {\em Communications in Mathematical Physics}, 360:555--609, 2018.

\bibitem[LW07]{liwang_navierstokes}
Xiaosheng Li and Jenn-Nan Wang.
\newblock Determination of viscosity in the stationary {N}avier-{S}tokes
  equations.
\newblock {\em J. Differential Equations}, 242(1):24--39, 2007.

\bibitem[MU18]{munozuhlmann}
Claudio Munoz and Gunther Uhlmann.
\newblock The {C}alder{\'o}n problem for quasilinear elliptic equations.
\newblock {\em arXiv:1806.09586}, 2018.

\bibitem[RR06]{renardy2006introduction}
Michael Renardy and Robert~C Rogers.
\newblock {\em An introduction to partial differential equations}, volume~13.
\newblock Springer Science \& Business Media, 2006.

\bibitem[Sal08]{salo2008calderon}
Mikko Salo.
\newblock Calder{\'o}n problem.
\newblock {\em Lecture Notes}, 2008.

\bibitem[Sal17]{salo2017normal}
Mikko Salo.
\newblock The {C}alder\'on problem and normal forms.
\newblock {\em arXiv preprint arXiv:1702.02136}, 2017.

\bibitem[SU87]{sylvester1987global}
John Sylvester and Gunther Uhlmann.
\newblock A global uniqueness theorem for an inverse boundary value problem.
\newblock {\em Annals of mathematics}, pages 153--169, 1987.

\bibitem[SU97]{sun1997inverse}
Ziqi Sun and Gunther Uhlmann.
\newblock Inverse problems in quasilinear anisotropic media.
\newblock {\em American Journal of Mathematics}, 119(4):771--797, 1997.

\bibitem[Sun96]{sun1996}
Ziqi Sun.
\newblock On a quasilinear inverse boundary value problem.
\newblock {\em Math. Z.}, 221(2):293--305, 1996.

\bibitem[Sun05]{sun2005}
Ziqi Sun.
\newblock Conjectures in inverse boundary value problems for quasilinear
  elliptic equations.
\newblock {\em Cubo}, 7(3):65--73, 2005.

\bibitem[Sun10]{sun2010inverse}
Ziqi Sun.
\newblock An inverse boundary-value problem for semilinear elliptic equations.
\newblock {\em Electronic Journal of Differential Equations (EJDE)[electronic
  only]}, 37:1--5, 2010.

\bibitem[SZ12]{salo2012inverse}
Mikko Salo and Xiao Zhong.
\newblock An inverse problem for the p-{L}aplacian: boundary determination.
\newblock {\em SIAM journal on mathematical analysis}, 44(4):2474--2495, 2012.

\bibitem[Tay11]{taylor2011partial}
Michael~E. Taylor.
\newblock {\em Partial differential equations {I}. {B}asic theory}, volume 115
  of {\em Applied Mathematical Sciences}.
\newblock Springer, New York, second edition, 2011.

\bibitem[Uhl09]{uhlmann2009calderon}
Gunther Uhlmann.
\newblock {E}lectrical impedance tomography and {C}alder\'on's problem.
\newblock {\em Inverse Problems}, 25:123011, 2009.

\bibitem[Vol14]{volpert2014elliptic}
Vitaly Volpert.
\newblock {\em Elliptic Partial Differential Equations: {V}olume 2:
  {R}eaction-Diffusion Equations}, volume 104.
\newblock Springer, 2014.

\bibitem[WZ19]{wang2016quadartic}
Yiran Wang and Ting Zhou.
\newblock Inverse problems for quadratic derivative nonlinear wave equations.
\newblock {\em Comm. PDE., to appear}, 2019.

\end{thebibliography}

\end{document}